\documentclass[12pt]{article}
\usepackage[latin1]{inputenc}
\usepackage{color}
\usepackage{bbm}
\usepackage{amsmath,amssymb}
\usepackage{amsthm}
\usepackage{enumerate}
\usepackage{mathrsfs}
\usepackage{verbatim}
\usepackage{graphicx}

\newtheorem{thm}{Theorem}
\newtheorem{cor}[thm]{Corollary}
\newtheorem{lem}[thm]{Lemma}
\newtheorem{prop}[thm]{Proposition}
\newtheorem{defin}[thm]{Definition}

\newtheorem{rem}[thm]{Remark}

\makeatletter
\newcommand{\wal}{{\rm wal}}

\newcommand{\icomp}{\mathtt{i}}

\newcommand{\rd}{\,\mathrm{d}}
\newcommand{\bsj}{\boldsymbol{j}}
\newcommand{\bsx}{\boldsymbol{x}}
\newcommand{\bsy}{\boldsymbol{y}}
\newcommand{\bsz}{\boldsymbol{z}}

\newcommand{\bsk}{\boldsymbol{k}}
\newcommand{\bsm}{\boldsymbol{m}}
\newcommand{\bst}{\boldsymbol{t}}
\newcommand{\bsa}{\boldsymbol{a}}
\newcommand{\bsc}{\boldsymbol{c}}
\newcommand{\bss}{\boldsymbol{s}}
\newcommand{\bsr}{\boldsymbol{r}}

\newcommand{\bszero}{\boldsymbol{0}}
\newcommand{\bsone}{\boldsymbol{1}}

\newcommand{\RR}{\mathbb{R}}
\newcommand{\EE}{\mathbb{E}}

\newcommand{\uu}{\mathfrak{u}}
\newcommand{\vv}{\mathfrak{v}}

\newcommand{\CC}{\mathbb{C}}
\newcommand{\FF}{\mathbb{F}}

\newcommand{\NN}{\mathbb{N}}

\newcommand{\DD}{\mathbb{D}}

\newcommand{\ZZ}{\mathbb{Z}}

\newcommand{\cP}{\mathcal{P}}

\newcommand{\vecs}{\boldsymbol{\sigma}}

\newcommand{\bsell}{\boldsymbol{\ell}}
\newcommand{\calD}{\mathcal{D}}
\newcommand{\calB}{\mathcal{B}}

\allowdisplaybreaks

\begin{document}
\title{Point sets with optimal order of extreme and periodic discrepancy}
\author{Ralph Kritzinger and Friedrich Pillichshammer\thanks{The first author is supported by the Austrian Science Fund (FWF), Project F5509-N26, which is a part of the Special Research Program ``Quasi-Monte Carlo Methods: Theory and Applications''.}}
\date{}

\maketitle

\begin{abstract}
We study the extreme and the periodic $L_p$ discrepancy of point sets in the $d$-dimensional unit cube. The extreme discrepancy uses arbitrary sub-intervals of the unit cube as test sets, whereas the periodic discrepancy is based on periodic intervals modulo one. This is in contrast to the classical star discrepancy, which uses exclusively intervals that are anchored in the origin as test sets. In a recent paper the authors together with Aicke Hinrichs studied relations between the $L_2$ versions of these notions of discrepancy and presented exact formulas for typical two-dimensional quasi-Monte Carlo point sets. In this paper we study the general $L_p$ case and deduce the exact order of magnitude of the respective minimal discrepancy in the number $N$ of elements of the considered point sets, for arbitrary but fixed dimension $d$, which is $(\log N)^{(d-1)/2}$.   
\end{abstract}

\centerline{\begin{minipage}[hc]{130mm}{
}
\end{minipage}}

\section{Introduction}

Let $\cP=\{\bsx_0,\bsx_1,\ldots,\bsx_{N-1}\}$ be an arbitrary $N$-element point set in the $d$-dimensional unit cube $[0,1)^d$.  For any measurable subset $B$ of $[0,1]^d$ the {\it counting function} $$A(B,\cP):=|\{n \in \{0,1,\ldots,N-1\} \ : \ \bsx_n \in B\}|$$ counts the number of elements from $\cP$ that belong to the set $B$. The {\it local discrepancy} of $\cP$ with respect to a given measurable ``test set'' $B$  is then given by $$A(B,\cP)-N \lambda (B),$$ where $\lambda$ denotes the Lebesgue measure of $B$. A global discrepancy measure is then obtained by considering a norm of the local discrepancy with respect to a fixed class of test sets. 

The classical {\it star $L_p$ discrepancy} uses as test sets exclusively the class of axis-parallel squares that are anchored in the origin. The formal definition is 
$$ L_{p,N}^{{\rm star}}(\cP):=\left(\int_{[0,1]^d}\left|A([\bszero,\bst),\cP)-N\lambda([\bszero,\bst))\right|^p\rd \bst\right)^{1/p},  $$
where for $\bst=(t_1,t_2,\ldots,t_d)\in [0,1]^d$ we set $[\bszero,\bst)=[0,t_1)\times [0,t_2)\times \ldots \times [0,t_d)$ with area $\lambda([\bszero,\bst))=t_1t_2\cdots t_d$.

The {\it extreme $L_p$ discrepancy} uses as test sets arbitrary axis-parallel rectangles contained in the unit square. For $\bsx=(x_1,x_2,\ldots,x_d)$ and $\bsy=(y_1,y_2,\ldots,y_d)$ in $[0,1]^d$ and $\bsx \leq \bsy$ let $[\bsx,\bsy)=[x_1,y_1)\times [x_2,y_2) \times \ldots \times [x_d,y_d)$, where $\bsx \leq \bsy$ means $x_j\leq y_j$ for all $j \in \{1,2,\ldots,d\}$. The extreme $L_p$ discrepancy of $\cP$ is then defined as
$$L_{p,N}^{\mathrm{extr}}(\cP):=\left(\int_{[0,1]^d}\int_{[0,1]^d,\, \bsx\leq \bsy}\left|A([\bsx,\bsy),\cP)-N\lambda([\bsx,\bsy))\right|^p\rd \bsx\rd\bsy\right)^{1/p}.  $$
Note that the only difference between standard and extreme $L_p$ discrepancy is the use of anchored and arbitrary rectangles in $[0,1]^d$, respectively. 

The {\it periodic $L_p$ discrepancy} uses periodic rectangles as test sets, which are defined as follows: For $x,y\in [0,1]$ set
$$ I(x,y)=\begin{cases}
           [x,y) & \text{if $x\leq y$}, \\
           [0,y)\cup [x,1)& \text{if $x>y$,}
          \end{cases}$$
and for $\bsx,\bsy$ as above we set $B(\bsx,\bsy)=I(x_1,y_1)\times I(x_2,y_2) \times \ldots \times I(x_d,y_d)$.
We define the periodic $L_p$ discrepancy of $\cP$ as
$$  L_{p,N}^{\mathrm{per}}(\cP):=\left(\int_{[0,1]^d}\int_{[0,1]^d}\left|A(B(\bsx,\bsy),\cP)-N\lambda(B(\bsx,\bsy))\right|^p\rd \bsx\rd\bsy\right)^{1/p}.  $$

All three notions of discrepancy are quantitative measures for the irregularity of distribution of point sets modulo one. Also, these discrepancies have close relations to the worst-case integration error of quasi-Monte Carlo rules in suitable reproducing kernel Hilbert spaces of functions over $[0,1]^d$; see, e.g., \cite{HKP20} and the references therein. In the case $p=2$ and $d=2$ these discrepancies have been studied in the recent paper \cite{HKP20}. The main results are exact formulas for the respective $L_2$ discrepancies for the Hammersley point set (see \cite[Theorem~8]{HKP20}) and for rational lattices (see \cite[Theorem~10]{HKP20}). 

\paragraph{Relations.} Note that due to the monotonicity of the $L_p$-norm we have for $1 \le p \le q < \infty$ for every $N$-element point set $\cP$ in $[0,1)^d$ that $$L_{p,N}^{\bullet}(\cP) \le L_{q,N}^{\bullet}(\cP),\quad \mbox{where $\bullet \in \{ {\rm star},{\rm extr},{\rm per}\}$.}$$

The following bounds and relations between star, extreme and periodic discrepancies are known (see \cite[Sections~1 and 2]{HKP20}):  For every $N$-element point set $\cP$ in $[0,1)^d$ we have 
$$L_{2,N}^{{\rm extr}}(\cP) \le L_{2,N}^{{\rm per}}(\cP) \quad \mbox{ and }\quad  L_{2,N}^{{\rm extr}}(\cP) \le L_{2,N}^{{\rm star}}(\cP).$$ The left relation can be easily generalized for arbitrary $p \ge 1$ to 
\begin{equation}\label{rel:exper}
L_{p,N}^{{\rm extr}}(\cP) \le L_{p,N}^{{\rm per}}(\cP).
\end{equation}
Furthermore, there exists a quantity $c_d>0$ such that for every $N$-element point set $\cP$ in $[0,1)^d$ we have 
\begin{equation}\label{lbdl2ext}
c_d (1+\log N)^{\frac{d-1}{2}} \le L_{2,N}^{{\rm extr}}(\cP)\le \left\{
\begin{array}{l}
 L_{2,N}^{{\rm per}}(\cP),\\[0.5em]
  L_{2,N}^{{\rm star}}(\cP).
\end{array}\right.
\end{equation}
It is also well-known that for every $p \in (1,\infty)$ there exist $N$-element point sets $\cP$ in $[0,1)^d$ with the optimal order of star $L_p$ discrepancy, i.e., with 
\begin{equation}\label{optordstdlp}
L_{p,N}^{{\rm star}}(\cP) \le C_{d,p} (1+\log N)^{\frac{d-1}{2}},
\end{equation}
where the positive quantity $C_{d,p}$ may only depend on $d$ and $p$ but is independent of $N$. Explicit constructions are known and comprise certain digital nets introduced by Chen and Skriganov~\cite{CS02} and Skriganov~\cite{S06} and higher order digital nets~\cite{DHMP1,DP14a,DP14b,Marnets}.

So for $p=2$ the minimal extreme $L_p$ discrepancy for $N$-element point sets in dimension $d$ is of exact order of magnitude $(1+\log N)^{\frac{d-1}{2}}$. In more detail 
\begin{equation}\label{optordL2extr}
\inf_{\cP \subseteq [0,1)^d \atop |\cP|=N} L_{2,N}^{{\rm extr}}(\cP) \asymp_{d} (1+\log N)^{\frac{d-1}{2}}.
\end{equation}
It is well-known that the same order of magnitude holds true for the star $L_p$ discrepancy for any $p \in [2,\infty)$, see, e.g., \cite{S06}. For the periodic discrepancy we only know the lower bound 
\begin{equation}\label{lbd:Lper}
\inf_{\cP \subseteq [0,1)^d \atop |\cP|=N} L_{2,N}^{{\rm per}}(\cP) \gtrsim_d (1+\log N)^{\frac{d-1}{2}};
\end{equation}
see, e.g., \cite[Corollary~2]{HKP20} or \cite{lev99}. A matching upper bound of the same order of magnitude in arbitrary but fixed dimension is still missing (for more on this we refer to Remark~\ref{re:lev} in Section~\ref{sec:perdisc}). 

\paragraph{Aim.} In this paper we continue our study of extreme and periodic discrepancy of point sets in dimension $d$. In Section~\ref{sec:extrdisc} we study the extreme $L_p$ discrepancy. The main tool is a Haar series approach in conjunction with the Littlewood-Paley inequality. We show that for $p\in (1,\infty)$ the exact order of the minimal extreme $L_p$ discrepancy for $N$-element point sets in dimension $d$ is again $(1+\log N)^{\frac{d-1}{2}}$ (see Theorem~\ref{thm:optordLpextr}).  In dimension $2$ we can even extend this result to $p=1$ (see Theorem~\ref{thm:l1} and Corollary~\ref{cor:optordLpextr}). Furthermore, as a byproduct, we find some easy proofs of already known relations between standard and extreme discrepancy without going a detour to numerical integration in function spaces as done in \cite{HKP20}. Finally, with the help of the Haar series expansion of the extreme discrepancy we can easily amplify our results about extreme discrepancy of $2$-dimensional point sets beyond the Hammersley net. We prove exact formulas for the extreme $L_2$ discrepancy of various digital nets in dimension 2 and show upper bounds on the extreme $L_p$ discrepancy for the same nets (see Theorem~\ref{2dnets}). The periodic discrepancy is studied in Section~\ref{sec:perdisc}. The lower bound of order of magnitude $(1+\log N)^{\frac{d-1}{2}}$ is easily established and follows directly from the corresponding result for the extreme $L_p$ discrepancy (see Corollary~\ref{cor:lowper}). Then we show that also the minimal periodic $L_2$ discrepancy ($p=2$) is of exact asymptotic order $(1+\log N)^{\frac{d-1}{2}}$ like the star and the extreme $L_p$ discrepancies and the optimal order of magnitude can be achieved by means of digitally shifted digital nets (see Theorem~\ref{thm12} and Corollary~\ref{co16}). The technical tool for this result is based on a Walsh expansion of the periodic $L_2$ discrepancy. A corresponding result beyond the case $p \le 2$ is still missing. A consequence for the diaphony is stated in Remark~\ref{re:dia}.

In Section~\ref{sec:sum} we briefly summarize the main results presented in this paper and we also collect some open problems.

The constructions of point sets studied in this paper are based on the concept of digitally shifted digital nets. We will recall the basic definitions in Section~\ref{sec:nets}. Readers who are familiar with this topic may move directly to Section~\ref{sec:extrdisc}.

\paragraph{Basic notations.} We denote by $\NN$ the set of positive integers and by $\NN_0$ the set of nonnegative integers, i.e., $\NN_0=\NN \cup \{0\}$. For $d \in \NN$ we write $[d]:=\{1,2,\ldots,d\}$. Subsets of $[d]$ are denoted by $\uu$ or $\vv$. The complement of a subset $\uu \subseteq [d]$ is $\uu^c= [d]\setminus \uu$.

For functions $f,g:D\subseteq\NN \rightarrow \RR^+$ we write $g(N) \lesssim f(N)$ (or $g(N) \gtrsim f(N)$),
if there exists a $C>0$ such that $g(N) \le C f(N)$ (or $g(N) \ge C f(N)$) for all $N \in D$. If we would like to stress that the quantity $C$ may also depend on other variables than $N$, say $\alpha_1,\ldots,\alpha_w$,
this will be indicated by writing $\lesssim_{\alpha_1,\ldots,\alpha_w}$ (or $\gtrsim_{\alpha_1,\ldots,\alpha_w}$). Furthermore, we also use $f(N) \asymp g(N)$ which means that $f(N) \lesssim g(N)$ and $f(N) \gtrsim g(N)$ simultaneously.

\section{Digital nets over $\ZZ_b$}\label{sec:nets}

The constructions of point sets considered in this paper are based on the concept of digital nets which goes back to Niederreiter~\cite{nie87,niesiam}. There is a whole theory on this type of constructions. Here we only present the most basic definitions needed for the present paper. More information about digital nets can be found in the books \cite{DP10, LeoPi,niesiam}.  

Throughout this paper let $b$ be a prime number and let $\ZZ_b$ denote the finite field with $b$ elements. We identify $\ZZ_b$ with the set $\{0,1,\ldots,b-1\}$ equipped with arithmetic operations modulo $b$.

\begin{defin}\label{def2}\rm
Let $d,m \in \NN$ and $t \in \{0,1,\ldots,m\}$. Choose $d$  $m \times m$ matrices $C_1,\ldots ,C_d$ over $\ZZ_b$ with the
following property: for any integers $\ell_1,\ldots ,\ell_d \in \NN_0$ with $\ell_1+\cdots+\ell_d=m-t$ the system of the 
\begin{center}
\begin{tabbing}
\hspace*{3cm}\=first \= $\ell_1$ rows of $C_1$, \=  together with the\\
\>\hspace{1cm}$\vdots$ \\
\> first \> $\ell_{d-1}$ rows of $C_{d-1}$, \=  together with the \\
\>first \> $\ell_d$ rows of $C_d$ 
\end{tabbing}
\end{center}
is linearly independent over $\ZZ_b$.

Consider the following construction principle for point sets consisting of $b^m$ points in $[0,1)^d$: represent $n \in \{0,1,\ldots,b^m -1\}$ in base $b$, $n=n_0+n_1 b+\cdots +n_{m-1} b^{m-1}$, where $n_j \in \ZZ_b$ for $j \in \{0,1,\ldots,m-1\}$, and multiply the matrix $C_j$, $j \in [d]$, with the digit-vector $\vec{n} = (n_0,\ldots,n_{m-1})^{\top}$,
\begin{eqnarray}
C_j \vec{n}=:(y_1^{(j)},\ldots ,y_m^{(j)})^{\top} \in \ZZ_b^m\nonumber.
\end{eqnarray} 
Now set 
\begin{eqnarray}
x_n^{(j)}:=\frac{y_1^{(j)}}{b}+ \cdots +\frac{y_m^{(j)}}{b^m} \nonumber
\end{eqnarray}
and
\begin{eqnarray}
\bsx_n = (x_n^{(1)}, \ldots ,x_n^{(d)}).\nonumber
\end{eqnarray}
The point set $\{\bsx_0,\bsx_1,\ldots,\bsx_{b^m-1}\}$ is called a {\it digital $(t,m,d)$-net over $\ZZ_b$} and the matrices $C_1,\ldots,C_d$ are called the {\it generating matrices} of the digital net.
\end{defin}

We will also consider digitally shifted digital nets. In the following we introduce the digital shift of depth $m$ for the one dimensional case. For higher dimensions each coordinate is shifted independently and therefore one just needs to apply the one dimensional shifting method to each coordinate independently. 

Let a point set $\cP =  \{x_0, \ldots, x_{N-1}\}$ in $[0,1)$ be given. Let $$x_n = \frac{x_{n,1}}{b} + \frac{x_{n,2}}{b^2} + \cdots+\frac{x_{n,m}}{b^m}$$ be the binary digit expansion of $x_{n}$. Choose digits $\sigma_{1}, \ldots, \sigma_{m} \in \ZZ_b$ i.i.d.  Then define $$z_{n,i} \equiv x_{n,i} + \sigma_{i} \pmod{b} \qquad \mbox{ for } i \in \{1,\ldots,m\}$$ with $z_{n,i} \in \ZZ_b$. Further, for $n \in \{0,1,\ldots, N-1\}$, choose $\delta_n  \in [0,\frac{1}{b^m})$ i.i.d. Then the randomized point set $\widetilde{\cP} = \{z_0,\ldots,z_{N-1}\}$ is given by $$z_{n} =  \frac{z_{n,1}}{b} + \cdots + \frac{z_{n,m}}{b^m} + \delta_n.$$ This means we apply the same digital shift to the first $m$ digits, whereas the following digits are shifted independently for each $x_{n}$. This kind of shift is sometimes referred to as a digital shift of depth $m$ (see, e.g., \cite{DP05}) and can be seen as a special case of the random $b$-ary digit-scrambling of depth $m$ from \cite[Section~2.4]{matou}. 

We will study the periodic $L_2$ discrepancy of digitally shifted digital nets in Section~\ref{sec:perdisc}. Sometimes, especially in dimension 2, the $\delta_n$ is not needed and it suffices to consider digitally shifted digital nets of the form $\widehat{\cP}=\{z_0,\ldots,z_{b^m-1}\}$ given by $$z_{n} =  \frac{z_{n,1}}{b} + \cdots + \frac{z_{n,m}}{b^m}.$$ Such nets will be also considered in Section~\ref{sec:extrdisc}.

\section{The extreme discrepancy}\label{sec:extrdisc}

Let $\cP$ be a set of $N$ points in $[0,1)^d$. For $\bsx\in[0,1]^d$ define
$$D(\bsx):=A([\bszero,\bsx),\cP)-N\lambda([\bszero,\bsx))$$ and for 
$\bsx,\bsy\in[0,1]^d$ with $\bsx\leq \bsy$ define
$$ \widetilde{D}(\bsx,\bsy)=A([\bsx,\bsy),\cP)-N\lambda([\bsx,\bsy)).$$

Obviously the interval $[\bsx,\bsy)$ can be written as a union and difference of certain intervals anchored in the origin. This leads to the following connection between the unanchored discrepancy function $\widetilde{D}$ and the anchored version $D$. To state this relation, we introduce some notation. For $\bsx=(x_1,\ldots,x_d)$, $\bsy=(y_1,\ldots,y_d)$ and $\uu \subseteq [d]$ we define the point $(\bsx_{\uu},\bsy_{\uu^c}):=(z_1,\ldots,z_d)$, where for $j \in [d]$ we put 
$$z_j=\left\{
\begin{array}{ll}
x_j & \mbox{ if $j \in \uu$,}\\
y_j & \mbox{ if $j \not\in \uu$.}
\end{array}\right.$$
Then, for $\bsx\leq \bsy$ we have
\begin{equation}\label{dtiledg}
    \widetilde{D}(\bsx,\bsy)=\sum_{\uu \subseteq [d]} (-1)^{|\uu|} 
            D((\bsx_{\uu},\bsy_{\uu^c})).
\end{equation}
We extend the right-hand side of \eqref{dtiledg} to arbitrary $\bsx,\bsy\in [0,1]^d$ (i.e., we do not demand $\bsx\leq \bsy$) and write 
$$g(\bsx,\bsy)=\sum_{\uu \subseteq [d]}(-1)^{|\uu|} D((\bsx_{\uu},\bsy_{\uu^c})).$$ For example, for $d=2$ we have $$g(x_1,x_2,y_1,y_2)=D(x_1,x_2)-D(x_1,y_2)-D(y_1,x_2)+D(y_1,y_2).$$ 

Note that for every $\vv \subseteq [d]$ we have 
\begin{equation}\label{relg}
(-1)^{|\vv|} g((\bsx_{\vv^c},\bsy_{\vv}),(\bsx_{\vv},\bsy_{\vv^c}))= g(\bsx,\bsy).
\end{equation}
 This can be seen in the following way: fix $\vv \subseteq [d]$ and take an arbitrary vector $(\bsx_{\uu},\bsy_{\uu^c})$ with $\uu \subseteq [d]$. Now exchange $|\vv|$ many components $x_j \leftrightarrow y_j$ for $j \in \vv$. Write $\vv=\vv_1 \cup \vv_2$ where $\vv_1 \subseteq \uu$ and $\vv_2 \subseteq \uu^c$. Then $(\bsx_\uu,\bsy_{\uu^c})$ in the definition of $g(\bsx,\bsy)$ corresponds to $$(\bsx_{\uu\setminus \vv_1 \cup \vv_2},\bsy_{\uu^c\setminus \vv_2 \cup \vv_1})$$ in the definition of $g((\bsx_{\vv^c},\bsy_{\vv}),(\bsx_{\vv},\bsy_{\vv^c}))$. But there might be a difference in the corresponding sign, since $D((\bsx_{\uu},\bsy_{\uu^c}))$ appears with the sign $(-1)^{|\uu|}$ in $g(\bsx,\bsy)$, whereas $D(\bsx_{\uu\setminus \vv_1 \cup \vv_2},\bsy_{\uu^c\setminus \vv_2 \cup \vv_1})$ has the sign $(-1)^{|\uu\setminus \vv_1 \cup \vv_2|}$ in the definition of $g((\bsx_{\vv^c},\bsy_{\vv}),(\bsx_{\vv},\bsy_{\vv^c}))$. However, we have in any case that $(-1)^{|\uu\setminus \vv_1 \cup \vv_2|}=(-1)^{|\uu| + |\vv|}$, because
\begin{eqnarray*}
|\uu\setminus \vv_1 \cup \vv_2| & = & |\uu\setminus \vv_1| + |\vv_2|\\
& = & |\uu|-|\vv_1|+|\vv_2|\\
& = & |\uu|-2 |\vv_1|+|\vv|\\
&\equiv & |\uu| + |\vv| \pmod{2}.
\end{eqnarray*} 
and from this we observe easily the relation \eqref{relg}.

From \eqref{relg} together with basic rules for integrals it follows that
\begin{eqnarray} \label{relate}
 L_{p,N}^{\mathrm{extr}}(\cP) & = & \left( \int_{[0,1]^d}\int_{[0,1]^d,\, \bsx\leq \bsy} |\widetilde{D}(\bsx,\bsy)|^p\rd\bsx\rd\bsy\right)^{1/p}\nonumber \\
 & = & \left(\frac{1}{2^d} \int_{[0,1]^d}\int_{[0,1]^d} |g(\bsx,\bsy)|^p\rd\bsx\rd\bsy\right)^{1/p}.
\end{eqnarray}   

Our aim is now to represent the extreme discrepancy in terms of Haar functions. We introduce the basic definitions in the following paragraphs. 

\paragraph{The Haar function system (in base $2$).} 

To begin with, a {\it dyadic interval} of length $2^{-j}, j\in {\mathbb N}_0,$ in $[0,1)$ is an interval of the form 
$$ I=I_{j,m}:=\left[\frac{m}{2^j},\frac{m+1}{2^j}\right) \ \ \mbox{for } \  m=0,1,\ldots,2^j-1.$$ We also define $I_{-1,0}:=[0,1)$.
The left and right half of $I=I_{j,m}$ are the dyadic intervals $I^+ = I_{j,m}^+ =I_{j+1,2m}$ and $I^- = I_{j,m}^- =I_{j+1,2m+1}$, respectively. For $j\in\NN_0$, the {\it Haar function} $h_I = h_{j,m}$ with support $I$ 
is the function on $[0,1)$ which is  $+1$ on the left half of $I$, $-1$ on the right half of $I$ and 0 outside of $I$. The $L_\infty$-normalized {\it Haar system} consists of
all Haar functions $h_{j,m}$ with $j\in{\mathbb N}_0$ and  $m\in \{0,1,\ldots,2^j-1\}$ together with the indicator function $h_{-1,0}$ of $[0,1)$.
Normalized in $L_2([0,1))$ we obtain the {\it orthonormal Haar basis} of $L_2([0,1))$. A crucial property of Haar functions we use in the following is
$$ \int_0^1 h_{j,m}(t)\rd t=\begin{cases}
                    1 & \text{if $j=-1$,} \\
                    0 & \text{if $j\in \NN_0$,}
                           \end{cases}$$

Let ${\mathbb N}_{-1}:=\{-1,0,1,2,\ldots\}$ and define ${\mathbb D}_j:=\{0,1,\ldots,2^j-1\}$ for $j\in{\mathbb N}_0$ and ${\mathbb D}_{-1}:=\{0\}$.
For $\bsj=(j_1,\dots,j_d)\in{\mathbb N}_{-1}^d$ and $\bsm=(m_1,\dots,m_d)\in {\mathbb D}_{\bsj} :={\mathbb D}_{j_1}\times \ldots \times {\mathbb D}_{j_d}$, 
the {\it Haar function} $h_{\bsj,\bsm}$
is given as the tensor product 
$$h_{\bsj,\bsm} (\bsx) := h_{j_1,m_1}(x_1)\, \cdots \, h_{j_d,m_d}(x_d) \ \ \ \mbox{ for } \bsx=(x_1,\dots,x_d)\in[0,1)^d.$$
The boxes $$I_{\bsj,\bsm} := I_{j_1,m_1} \times \ldots \times I_{j_d,m_d}$$ are called {\it dyadic boxes}.

The $L_\infty$-normalized tensor {\it Haar system} consists of all Haar functions $h_{\bsj,\bsm}$ with $\bsj\in{\mathbb N}_{-1}^d$ and  $\bsm \in {\mathbb D}_{\bsj}$. Normalized in $L_2([0,1)^d)$ we obtain the {\it orthonormal Haar basis} of $L_2([0,1)^d)$. 
  
Let $f\in L_2([0,1]^d)$. Then Parseval's identity states that $$ \|f\|_{L_2([0,1]^d)}=\sum_{\bsj \in \NN_{-1}^d} 2^{-|\bsj|}\sum_{\bsm\in\DD_{\bsj}}|\mu_{\bsj,\bsm}(f)|^2, $$  where for $\bsj\in\NN_{-1}^d$ and $\bsm\in\DD_{\bsj}$ the numbers
  $$ \mu_{\bsj,\bsm}(f):=\langle f,h_{\bsj,\bsm}\rangle=\int_{[0,1]^d} f(\bsx)h_{\bsj,\bsm}(\bsx)\rd\bsx $$  are called the Haar coefficients of $f$.

\paragraph{The Haar coefficients of the function $g$.}

Let $\bsj=(j_1,\dots,j_d,k_1,\dots,k_d)\in \NN_{-1}^{2d}$ and $\bsm=(m_1,\dots,m_{2d})\in\DD_{\bsj}$. We compute the Haar coefficient $\mu_{\bsj,\bsm}(g)$ of $g$ which is defined as
  $$ \mu_{\bsj,\bsm}(g):=\int_{[0,1]^{d}}\int_{[0,1]^{d}} g(\bsx,\bsy)h_{\bsj,\bsm}(\bsx,\bsy)\rd \bsx\rd\bsy, $$
  where $g(\bsx,\bsy)=g(x_1,\dots,x_d,y_1,\dots,y_d)$ and $$h_{\bsj,\bsm}(\bsx,\bsy)=h_{j_1,m_1}(x_1)\cdots h_{j_d,m_d}(x_d)h_{k_1,m_{d+1}}(y_1)\cdots h_{k_d,m_{2d}}(y_d).$$
  
\begin{lem} \label{mainlemma}
Let $\bsj\in\NN_{-1}^{2d}$ such that exactly $d$ components of $\bsj$ are $-1$ and more precisely we have the following: for every $i\in[d]$ we either have $j_i=-1$ or $k_i=-1$. Let $\bsj'=(j_1',\dots,j_d')\in\NN_0^d$ such that for all $i\in[d]$ we have $j_i'=\max\{j_i,k_i\}$ (i.e. it is the number $j_i$ or $k_i$ which is not -1). Let further $\bsm'=(m_1',\dots,m_d')$ be such that $m_i'=m_i$ if $j_i'=j_i$ and $m_i'=m_{i+d}$ otherwise. Then we have $$|\mu_{\bsj,\bsm}(g)|=|\mu_{\bsj',\bsm'}(D)|.$$ In all other cases for $\bsj\in\NN_{-1}^{2d}$ we have $\mu_{\bsj,\bsm}(g)=0$.
\end{lem}
  
\begin{proof}
Let $\bsj\in\NN_{-1}^{2d}$. Assume that $\kappa\in\{0,1,\dots,2d\}$ components of $\bsj$ are $-1$. We have $$ \mu_{\bsj,\bsm}(g)=\sum_{\uu \subseteq [d]} (-1)^{|\uu|}\left\langle D((\bsx_{\uu},\bsy_{\uu^c})),h_{\bsj,\bsm} \right\rangle,$$ where $$\left\langle D((\bsx_{\uu},\bsy_{\uu^c})),h_{\bsj,\bsm} \right\rangle = \int_{[0,1]^d} \int_{[0,1]^d} D((\bsx_{\uu},\bsy_{\uu^c})) h_{\bsj,\bsm}(\bsx,\bsy) \rd \bsx \rd \bsy.$$
    
It is clear that for $\kappa<d$ we have $\left\langle D((\bsx_{\uu},\bsy_{\uu^c})),h_{\bsj,\bsm} \right\rangle=0$ for all $\uu \subseteq [d]$. 

Next, let $\kappa>d$. Then, according to the  pigeon-hole principle there exists an $i\in [d]$ such that $j_i=k_i=-1$. Fix such an index $i$. Then we have
\begin{align*}
\mu_{\bsj,\bsm}(g)=&\sum_{\uu \subseteq [d]} (-1)^{|\uu|}\left\langle D((\bsx_{\uu},\bsy_{\uu^c})),h_{\bsj,\bsm} \right\rangle \\
=& \sum_{\uu \subseteq [d] \setminus \{i\}} \left[(-1)^{|\uu|+1}\left\langle D((\bsx_{\uu \cup \{i\}},\bsy_{\uu^c})),h_{\bsj,\bsm} \right\rangle +(-1)^{|\uu|}\left\langle D((\bsx_{\uu},\bsy_{\uu^c \cup \{i\}})),h_{\bsj,\bsm} \right\rangle \right]\\
=& 0,
\end{align*}
since obviously
\begin{align*}
\left\langle D((\bsx_{\uu \cup \{i\}},\bsy_{\uu^c})),h_{\bsj,\bsm} \right\rangle = \left\langle D((\bsx_{\uu},\bsy_{\uu^c \cup \{i\}})),h_{\bsj,\bsm} \right\rangle.
\end{align*}
This can be immediately seen observing that the integral over $y_i$ in the first Haar coefficient has no effect (delivers a factor 1). The same is true for $x_i$ in the second Haar coefficient. Now relabelling the variable $y_i$ to $x_i$ in the second integral yields the above identity. 

Finally, consider the case that $\kappa=d$. Let $\bsj\in\NN_{-1}^{2d}$ be such that for every $i\in [d]$ we either have $j_i=-1$ or $k_i=-1$. Then there is exactly one subset $\uu \subseteq [d]$ such that $\left\langle D((\bsx_{\uu},\bsy_{\uu^c})),h_{\bsj,\bsm} \right\rangle$ is not (necessarily) zero. This particular set $\uu$ consists of exactly those indices $i \in [d]$ for which $k_i=-1$, i.e. $$\uu=\{i \in [d]\ : \ k_i=-1\}.$$  For this set $\uu$ we have $$ \left\langle D((\bsx_{\uu},\bsy_{\uu^c})),h_{\bsj,\bsm} \right\rangle=\mu_{\bsj',\bsm'}(D) $$ with $\bsj'$ and $\bsm'$ given as above and hence $$\mu_{\bsj,\bsm}(g) = (-1)^{|\uu|} \mu_{\bsj',\bsm'}(D).$$ From this the result follows and the proof is complete.
\end{proof}
  
\paragraph{Haar representation of the extreme $L_p$ discrepancy.}  Lemma~\ref{mainlemma} yields the following results on the extreme $L_p$ discrepancy of point sets in $[0,1)^d$ which is crucial for most other results on extreme discrepancy in this section.
  
\begin{prop} \label{prop2}
Let $\cP$ be an $N$-element point set in $[0,1)^d$ and $D$ its anchored discrepancy function with Haar coefficients $\mu_{\bsj,\bsm}=\mu_{\bsj,\bsm}(D)$ for $\bsj\in\NN_{-1}^d$ and $\bsm\in\DD_{\bsj}$. Then we have
\begin{equation}\label{l2exthaar}
(L_{2,N}^{\mathrm{extr}}(\cP))^2=\sum_{\bsj \in \NN_0^d} 2^{|\bsj|}\sum_{\bsm\in\DD_{\bsj}}|\mu_{\bsj,\bsm}|^2.
\end{equation}
For arbitrary $p \in (1,\infty)$ we have  
\begin{equation}\label{lpexthaar}
L_{p,N}^{\mathrm{extr}}(\cP)\asymp_{p,d} \left\|\left(\sum_{\bsj \in \NN_0^d} \sum_{\bsm \in \mathbb{D}_{\bsj}} 2^{2|\bsj|}|\mu_{\bsj,\bsm}|^2 \, {\mathbf 1}_{I_{\bsj,\bsm}}\right)^{1/2}\right\|_{L_p([0,1]^{d})},
\end{equation}
where the implied factors depend only on $p$ and $d$, but not on $N$.
\end{prop}
  
\begin{proof}
First we prove the formula \eqref{l2exthaar} for the extreme $L_2$ discrepancy.  By~\eqref{relate}, Parseval's identity and Lemma~\ref{mainlemma} we have
     \begin{align*}
 (L_{2,N}^{\mathrm{extr}}(\cP))^2=&\frac{1}{2^d}\|g\|_{L_2([0,1]^{2d})}^2\\
         =&\frac{1}{2^d} \sum_{\bsj \in \NN_{-1}^{2d}} 2^{|\bsj|}\sum_{\bsm\in\DD_{\bsj}}|\mu_{\bsj,\bsm}(g)|^2 \\
         =&\frac{1}{2^d} \sum_{\uu \subseteq [d]} \sum_{\bsj'=(j_1',\dots,j_d')\in\NN_0^d} \sum_{\substack{\bsj=(j_1,\dots,j_d,k_1,\dots,k_d)\in\NN_{-1}^{2d} \\ \text{for\, } i \in [d]: \\ (j_i,k_i)=(j_i',-1) \text{\,if \,} i \not\in \uu \\  (j_i,k_i)=(-1,j_i') \text{\,if \,} i \in \uu}}2^{|\bsj'|}\sum_{\bsm'\in\DD_{\bsj'}}|\mu_{\bsj',\bsm'}(D)|^2 \\
            =&\frac{1}{2^d} \sum_{\uu \subseteq [d]} \sum_{\bsj'\in\NN_0^d} 2^{|\bsj'|}\sum_{\bsm'\in\DD_{\bsj'}}|\mu_{\bsj',\bsm'}|^2 \\
            =&\sum_{\bsj'\in\NN_0^d} 2^{|\bsj'|}\sum_{\bsm'\in\DD_{\bsj'}}|\mu_{\bsj',\bsm'}|^2.
     \end{align*}
     This completes the proof of \eqref{l2exthaar}.
 
In order to prove formula \eqref{lpexthaar} for the extreme $L_p$ discrepancy we make use of the Littlewood-Paley inequality which provides a tool which can be used to replace Parseval's equality and Bessel's inequality for functions in $L_p([0,1]^d)$ with $p \in (1,\infty)$. It involves the square function $S(f)$ of a function $f\in L_p([0,1]^d)$ which is given as
$$ S(f) = \left( \sum_{\bsj \in \NN_{-1}^d} \sum_{\bsm \in \mathbb{D}_{\bsj}} 2^{2|\bsj|} \, \langle f , h_{\bsj,\bsm} \rangle^2 \, {\mathbf 1}_{I_{\bsj,\bsm}} \right)^{1/2},$$ 
where ${\mathbf 1}_I$ is the indicator function of $I$.
Let $p \in (1,\infty)$ and let $f\in L_p([0,1]^d)$. Then the Littlewood-Paley inequality states $$ \| S(f) \|_{L_p} \asymp_{p,d} \| f \|_{L_p}.$$    
By~\eqref{relate}, the Littlewood-Paley inequality and Lemma~\ref{mainlemma} we have similarly as in the proof of \eqref{l2exthaar} that 
\begin{eqnarray*}
L_{p,N}^{\mathrm{extr}}(\cP) & = & \frac{1}{2^{d/p}}\|g\|_{L_p([0,1]^{2d})} \asymp_{p,d} \|S(g)\|_{L_p([0,1]^{2d})} \\
&=&\left\|\left(\sum_{\bsj \in \NN_{-1}^{2d}} \sum_{\bsm \in \mathbb{D}_{\bsj}} 2^{2|\bsj|}|\mu_{\bsj,\bsm}(g)|^2 \, {\mathbf 1}_{I_{\bsj,\bsm}}\right)^{1/2}\right\|_{L_p([0,1]^{2d})}\\
&=&\left\|\left(2^d\sum_{\bsj' \in \NN_0^d} \sum_{\bsm' \in \mathbb{D}_{\bsj'}} 2^{2|\bsj'|}|\mu_{\bsj',\bsm'}(D)|^2 \, {\mathbf 1}_{I_{\bsj',\bsm'}}\right)^{1/2}\right\|_{L_p([0,1]^{d})} \\
 &\asymp_d&\left\|\left(\sum_{\bsj' \in \NN_0^d} \sum_{\bsm' \in \mathbb{D}_{\bsj'}} 2^{2|\bsj'|}|\mu_{\bsj',\bsm'}|^2 \, {\mathbf 1}_{I_{\bsj',\bsm'}}\right)^{1/2}\right\|_{L_p([0,1]^{d})}.
\end{eqnarray*}     
This completes the proof of Proposition~\ref{prop2}.
\end{proof}

We deduce two corollaries from Proposition~\ref{prop2}. First, one can see in the same way as for the star $L_p$ discrepancy, that also the extreme $L_p$ discrepancy satisfies a general lower bound of Roth-type.

\begin{cor}\label{co0}
Let $p >1$. For every $N$-element point set $\cP$ in $[0,1)^d$ we have $$L_{p,N}^{{\rm extr}}(\cP) \gtrsim_{p,d} (1+\log N)^{\frac{d-1}{2}}.$$
\end{cor} 
 
\begin{proof}
Starting from Eq. \eqref{lpexthaar} in Proposition~\ref{prop2} the proof follows exactly the lines of \cite[Proof of Theorem~3.6]{DHP}.
\end{proof}

Since  $$  (L_{2,N}^{{\rm star}}(\cP))^2=\sum_{\bsj \in \NN_{-1}^d} 2^{|\bsj|}\sum_{\bsm\in\DD_{\bsj}}|\mu_{\bsj,\bsm}|^2,$$ we immediately obtain $L_{2,N}^{\mathrm{extr}}(\cP)\leq L_{2,N}^{{\rm star}}(\cP)$ for every point set $\cP$ in $[0,1]^d$ as recently discovered in~\cite{HKP20} by relating these discrepancy notions to the worst-case integration errors in certain Hilbert spaces. Now we can even extend this result to some extent to general $p \in (1,\infty)$.

\begin{cor}\label{co1}
For every $p \in (1,\infty)$ and for every $N$-element point set $\cP$ in $[0,1)^d$ we have $$L_{p,N}^{\mathrm{extr}}(\cP) \lesssim_{p,d} L_{p,N}^{\mathrm{star}}(\cP).$$
\end{cor}

\begin{proof}
Obviously,
\begin{eqnarray*}
0 & \le & \left(\sum_{\bsj \in \NN_0^d} \sum_{\bsm \in \mathbb{D}_{\bsj}} 2^{2|\bsj|}|\mu_{\bsj,\bsm}|^2 \, {\mathbf 1}_{I_{\bsj,\bsm}}\right)^{1/2}\\
& \le & \left(\sum_{\bsj \in \NN_{-1}^d} \sum_{\bsm \in \mathbb{D}_{\bsj}} 2^{2|\bsj|}|\mu_{\bsj,\bsm}|^2 \, {\mathbf 1}_{I_{\bsj,\bsm}}\right)^{1/2}.
\end{eqnarray*}
Since for $L_p$-functions $0 \le f \le g$ we always have $\|f\|_{L_p} \le \|g\|_{L_p}$ we obtain
\begin{eqnarray*}
L_{p,N}^{\mathrm{extr}}(\cP) & \lesssim_{p,d} & \left\|\left(\sum_{\bsj \in \NN_0^d} \sum_{\bsm \in \mathbb{D}_{\bsj}} 2^{2|\bsj|}|\mu_{\bsj,\bsm}|^2 \, {\mathbf 1}_{I_{\bsj,\bsm}}\right)^{1/2}\right\|_{L_p([0,1]^{d})}\\
& \le & \left\|\left(\sum_{\bsj \in \NN_{-1}^d} \sum_{\bsm \in \mathbb{D}_{\bsj}} 2^{2|\bsj|}|\mu_{\bsj,\bsm}|^2 \, {\mathbf 1}_{I_{\bsj,\bsm}}\right)^{1/2}\right\|_{L_p([0,1]^{d})}\\
& \lesssim_{p,d} & L_{p,N}^{\mathrm{star}}(\cP).
\end{eqnarray*}
Here the first estimate follows from Proposition~\ref{prop2} and the final estimate follows from the corresponding and well-known ``Littlewood-Paley'' result for the standard $L_p$ discrepancy (see, e.g., \cite{HKP14}). This completes the proof.
\end{proof}

As a consequence, all point sets and sequences with the optimal order of star $L_p$ discrepancy also achieve the optimal order of extreme $L_p$ discrepancy as well. This means that we can extend Eq. \eqref{optordL2extr} to general $p \in (1,\infty)$. Combining \eqref{optordstdlp}, Corollary~\ref{co1} and the lower bound from Corollary~\ref{co0} we obtain:

\begin{thm}\label{thm:optordLpextr}
For every $p \in (1,\infty)$ and every $d,N \in \NN$ we have $$\inf_{\cP \subseteq [0,1)^d \atop |\cP|=N} L_{p,N}^{{\rm extr}}(\cP) \asymp_{d,p} (1+\log N)^{\frac{d-1}{2}}.$$ 
\end{thm}
 
\paragraph{The  extreme $L_1$ discrepancy for dimension two.} In dimension $d=2$ we can prove the lower bound even for $p=1$. For the proof we also use a method introduced by Hal\'{a}sz~\cite{hala} for the star $L_1$ discrepancy.

According to \eqref{relate} the extreme $L_1$ discrepancy of $\cP$ in $[0,1)^2$ can be calculated via
 $$ L_{1,N}^{\mathrm{extr}}(\cP)=\frac14 \|g\|_{L_1([0,1]^4)}, $$
 where
 $$ g(x_1,x_2,y_1,y_2)=D(x_1,x_2)-D(x_1,y_2)-D(y_1,x_2)+D(y_1,y_2). $$
From the proof of Lemma~\ref{mainlemma} we find the following crucial observation: Let $\bsj=(j_1,j_2,-1,-1)$ with $j_1,j_2\in\NN_0$ and $\bsm=(m_1,m_2,0,0)\in\DD_{\bsj}$. Set $\bsj'=(j_1,j_2)$ and $\bsm'=(m_1,m_2)\in\DD_{\bsj'}$.
Then we have
\begin{eqnarray} \label{cruc}
\langle g,h_{\bsj,\bsm}\rangle  & = &\int_{[0,1]^4}(D(x_1,x_2)-D(x_1,y_2)-D(y_1,x_2)+D(y_1,y_2)) \nonumber\\
& & \hspace{1cm}\times h_{j_1,m_1}(x_1)h_{j_2,m_2}(x_2)h_{-1,0}(y_1)h_{-1,0}(y_2)\rd (x_1, x_2, y_1, y_2) \nonumber\\
&=& \int_{[0,1]^2}D(x_1,x_2)h_{j_1,m_1}(x_1)h_{j_2,m_2}(x_2)\rd (x_1, x_2)\nonumber \\
&=& \langle D,h_{\bsj',\bsm'}\rangle.
 \end{eqnarray}
 Four-dimensional Haar functions of this form clearly behave like two-dimensional Haar functions, which is important for the product rule to hold (which says that the product of intersecting two-dimensional Haar functions is a Haar function again.)
Based on this observation we can prove the following lower bound on the extreme $L_1$ discrepancy of two-dimensional point sets.

\begin{thm} \label{thm:l1}
There exists a constant $c>0$ such that for every $N$-element point set $\cP$ in the unit square we have $$ L_{1,N}^{\mathrm{extr}}(\cP)\geq c\sqrt{\log{N}}. $$
 \end{thm}
 
\begin{proof}
Let $n\in\NN$ such that $2^{n-1}\leq 2N\leq 2^{n}$. For $k \in \{0,1,\dots,n\}$ we introduce the functions $$ f_k:=\sum_{\bsj=(k,n-k,-1,-1)}\sum_{\bsm\in\DD_{\bsj}}\varepsilon_{\bsj,\bsm}h_{\bsj,\bsm}. $$ For $\bsj=(k,n-k,-1,-1)$ and $\bsm=(m_1,m_2,0,0)\in\DD_{\bsj}$ we set $\bsj'=(k,n-k)$ and $\bsm'=(m_1,m_2)\in\DD_{\bsj'}$. The signs $\varepsilon_{\bsj,\bsm}$ are chosen such that $\varepsilon_{\bsj,\bsm}=-1$ if $\cP\cap I_{\bsj',\bsm'}=\emptyset$ and $\varepsilon_{\bsj,\bsm}=0$ otherwise. Then we have (using~\eqref{cruc})
\begin{align} \label{lowerf}
\langle g, f_k\rangle \nonumber =&-\sum_{\bsj=(k,n-k,-1,-1)}\sum_{\substack{\bsm\in\DD_{\bsj}\\ \cP\cap I_{\bsj',\bsm'}=\emptyset }}\langle g, h_{\bsj,\bsm}\rangle     =-\sum_{\bsj'=(k,n-k)}\sum_{\substack{\bsm'\in\DD_{\bsj'}\\ \cP\cap I_{\bsj',\bsm'}=\emptyset }}\langle D, h_{\bsj',\bsm'}\rangle \\
=&\sum_{\bsj'=(k,n-k)}\sum_{\substack{\bsm'\in\DD_{\bsj'}\\ \cP\cap I_{\bsj',\bsm'}=\emptyset }}N2^{-2|\bsj'|-4} \geq 2^{n-2}2^{n-1}2^{-2n-4}=2^{-7}. 
\end{align}
Here we used that at least $2^{n-1}$ of the $2^n$ boxes $I_{\bsj',\bsm'}$ must be empty and the well-known and easily checked fact that $\langle D, h_{\bsj',\bsm'}\rangle=-N2^{-2|\bsj'|-4}$ if $I_{\bsj',\bsm'}$ is empty. 
    
The rest of the proof works completely the same as Halasz' proof on the standard $L_1$ discrepancy of two-dimensional point sets from \cite{hala}. We introduce the Riesz product $$F:=\prod_{k=0}^n \left(1+\frac{\icomp \gamma}{\sqrt{n+1}} f_k\right)-1=\frac{\icomp \gamma}{\sqrt{n+1}}\sum_{k=0}^{n}f_k+F_{>n},$$  where $F_{>n}=F_2+\dots+F_n $ with $$F_k=\left(\frac{\icomp \gamma}{\sqrt{n+1}}\right)^k \sum_{0\leq l_1<\dots<l_k\leq n} f_{l_1}\cdots f_{l_k}$$ for $k \in \{0,1,\dots,n\}$. Clearly, with~\eqref{lowerf} we have $$ |\langle g,F\rangle| \geq  \frac{ \gamma}{\sqrt{n+1}}\left|\sum_{k=0}^{n}\langle g, f_k\rangle \right|-|\langle g,F_{>n}\rangle|\geq \frac{\gamma}{2^7}\sqrt{n+1}-|\langle g,F_{>n}\rangle|. $$

For $\bsr=(r_1,r_2)\in \NN_0^2$ with $r_1+r_2=r$ we call a function of the form $$f_{\bsr}:=\sum_{\bsj=(r_1,r_2,-1,-1)}\sum_{\bsm\in\DD_{\bsj}}\varepsilon_{\bsj,\bsm}h_{\bsj,\bsm}$$ with coefficients $\varepsilon_{\bsj,\bsm}\in\{-1,0,1\}$ an $r$-function. (Therefore the $f_k$ from above are $r$-functions with $\bsr=(k,n-k,-1,-1)$ and $r=n$.) For an $r$-function we have
\begin{align} \label{estimate}
|\langle g,f_{\bsr}\rangle| \leq &\sum_{\bsj=(r_1,r_2,-1,-1)}\sum_{\bsm\in\DD_{\bsj}}|\langle g, h_{\bsj,\bsm}\rangle|\nonumber\\
= & \sum_{\bsj'=(r_1,r_2)}\sum_{\bsm'\in\DD_{\bsj'}}|\langle D, h_{\bsj',\bsm'}\rangle|  \nonumber\\ \nonumber
\leq& \sum_{\bsj'=(r_1,r_2)}\sum_{\bsm'\in\DD_{\bsj'}}\left|\left\langle \sum_{\bsz\in\cP}\bsone_{[\bszero,\bsx)}(\bsz), h_{\bsj',\bsm'}\right\rangle\right|+\sum_{\bsj'=(r_1,r_2)}\sum_{\bsm'\in\DD_{\bsj'}}|\langle Nx_1x_2, h_{\bsj',\bsm'}\rangle| \\  \nonumber
\leq & \sum_{\bsj'=(r_1,r_2)}\sum_{\bsz\in\cP}\sum_{\substack{\bsm'\in\DD_{\bsj'}\\ \bsz\in \cP\cap I_{\bsj',\bsm'}}}\underbrace{|\langle \bsone_{[\bszero,\bsx)}(\bsz), h_{\bsj',\bsm'}\rangle|}_{\leq 2^{-|\bsj'|}}+\sum_{\bsj'=(r_1,r_2)}\sum_{\bsm'\in\DD_{\bsj'}} \frac{N}{2^{2|\bsj'|+4}} \nonumber \\
\leq & \frac{N}{2^{|\bsj'|}}+\frac{2^{|\bsj'|}N}{2^{2|\bsj'|+4}}\nonumber\\
\lesssim & \frac{N}{2^r}.
\end{align}

Now, by the fact that products of two-dimensional Haar functions of same size which intersect are Haar functions again and since four-dimensional Haar functions with $\bsj=(j_1,j_2,-1,-1)$ behave like two-dimensional Haar functions, we find that $f_{l_1}\cdots f_{l_k}$ for $0\leq l_1<\dots<l_k\leq n$ is an $r$-function again, with corresponding vector $\bss=(n-l_1,l_k,-1,-1)$, i.e. $s=n-l_1+l_k$. The parameter $s$ can have the values $n+1,n+2,\dots, 2n$. A term $f_{l_1}\cdots f_{l_k}$ is an $r$-function with parameter $s$ if we have $n-l_1+l_k=s$ and $l_k\leq n$; i.e., if there are $2n-s+1$ possible choices for $l_1$ to guarantee both conditions ($l_1 \in\{0,1,\dots, 2n-s\}$). For the remaining $(k-2)$ indices $l_2,\dots,l_{k-1}$ we have $s-n-1$ possible values they can take, which can be selected in $\binom{s-n-1}{k-2}$ many ways.

Choose $\gamma\leq\frac12$. Then we obtain (using~\eqref{estimate})
\begin{align*}
|\langle g,F_{>n}\rangle| \le & \sum_{k=2}^n |\langle g, F_k\rangle|\\
\le & \sum_{k=2}^n \left(\frac{\gamma}{\sqrt{n+1}}\right)^k \sum_{0 \le l_1< \cdots < l_k \le n} |\langle g,f_{l_1} \cdots f_{l_k}\rangle|\\ 
\le & \sum_{s=n+1}^{2 n} \sum_{k=2}^n \left(\frac{\gamma}{\sqrt{n+1}}\right)^k \sum_{0 \le l_1< \cdots < l_k \le n \atop s=n-l_1+l_k} \frac{N}{2^s}\\
\le & \sum_{s=n+1}^{2n}(2n-s+1)\sum_{k=2}^{s-n+1} \binom{s-n-1}{k-2} \left(\frac{\gamma}{\sqrt{n+1}}\right)^k \frac{N}{2^s} \\
\leq &\left(\frac{\gamma}{\sqrt{n+1}}\right)^2n\sum_{s=n+1}^{2n}\sum_{k=0}^{s-n-1}  \binom{s-n-1}{k} \left(\frac{\gamma}{\sqrt{n+1}}\right)^k \frac{N}{2^s} \\
=&\frac{\gamma^2n}{n+1}\sum_{s=n+1}^{2n} \left(1+\frac{\gamma}{\sqrt{n+1}}\right)^{s-n-1} \frac{N}{2^s} \\
=&\frac{\gamma^2n}{n+1}\sum_{s=n+1}^{2n} \left(\frac12+\frac{\gamma}{2\sqrt{n+1}}\right)^{s-n-1} \frac{N}{2^{n+1}} \\
\leq& \gamma^2 \sum_{s=n+1}^{\infty} \left(\frac12+\frac{\gamma}{2\sqrt{n+1}}\right)^{s-n-1}\\
\leq& \gamma^2 \sum_{s=n+1}^{\infty} \left(\frac34\right)^{s-n-1}\\
=& 4\gamma^2\\
\leq&1.
     \end{align*}
 Therefore $|\langle g, F\rangle| \gtrsim \sqrt{n+1}$.
     On the other hand, we also have
     $$ \|F\|_{L_{\infty}([0,1]^4)}\leq \left(1+\frac{\gamma^2}{n+1}\right)^{\frac{n+1}{2}}+1 \leq \mathrm{e}^{\frac{\gamma^2}{2}}+1\lesssim 1. $$
     Therefore,
     $$ L_{1,N}^{\mathrm{extr}}(\cP)=\frac14 \|g\|_{L_4([0,1]^4)}\gtrsim \frac{|\langle g, F\rangle|}{\|F\|_{L_{\infty}([0,1]^4)}} \gtrsim \sqrt{n+1} \gtrsim \sqrt{\log{N}}. $$ 
 \end{proof}

As a consequence we obtain the following corollary:

\begin{cor}\label{cor:optordLpextr}
For every $p \in [1,\infty)$ we have $$\inf_{\cP \subseteq [0,1)^2 \atop |\cP|=N} L_{p,N}^{{\rm extr}}(\cP) \asymp_{d,p} (1+\log N)^{1/2}.$$ 
\end{cor}

\paragraph{Extreme discrepancy of digital nets in dimension 2.}

Proposition~\ref{prop2} demonstrates that a calculation of the Haar coefficients of the anchored discrepancy function of a point set yields not only results on the standard $L_p$ discrepancy, but also on the extreme $L_p$ discrepancy. For the latter it even suffices to evaluate only those coefficients where $\bsj\in\NN_0^d$.
The Haar coefficients of the discrepancy function of certain digital nets have been computed exactly in~\cite{Kritz}. We use these results to obtain exact formulas for the extreme $L_2$ discrepancy of these nets. We introduce the relevant nets. 

\begin{itemize}
\item We study digital $(0,m,2)$-nets generated by the following $m\times m$ matrices over $\ZZ_2$:
\begin{equation} \label{matrixa} C_1=
\begin{pmatrix}
0 & 0 & 0 & \cdots & 0 & 0 & 1 \\
0 & 0 & 0 & \cdots & 0 & 1 & 0 \\
0 & 0 & 0 & \cdots & 1 & 0 & 0 \\
\vdots & \vdots & \vdots & \ddots & \vdots & \vdots & \vdots & \\
0 & 0 & 1 & \cdots & 0 & 0 & 0 \\
0 & 1 & 0 & \cdots & 0 & 0 & 0 \\
1 & 0 & 0 & \cdots & 0 & 0 & 0 \\
\end{pmatrix}
\text{\, and\, \,}
 C_2=
\begin{pmatrix}
1 & 0 & 0 & \cdots & 0 & 0 & a_{1} \\
0 & 1 & 0 & \cdots &  0 & 0 & a_{2} \\
0 & 0 & 1 & \cdots & 0 & 0 & a_{3} \\
\vdots & \vdots & \vdots & \ddots & \vdots & \vdots & \vdots & \\
0 & 0 & 0 & \cdots &  1 & 0 & a_{m-2} \\
0 & 0 & 0 & \cdots &  0 & 1 & a_{m-1} \\
0 & 0 & 0 & \cdots &  0 & 0 & 1 \\
\end{pmatrix}.
\end{equation}
We study the discrepancy of the digital net $\cP_{\bsa}(\vecs)$ with $\bsa=(a_1,\dots,a_{m-1})^{\top} \in \ZZ_2^{m-1}$, generated by $C_1$ and $C_2$ and digitally shifted by $\vecs=(\sigma_1,\dots,\sigma_m)^{\top}\in \ZZ_2^m$. The set $\cP_{\bsa}(\vecs)$ can be written as
$$ \cP_{\bsa}(\vecs)=\left\{\bigg(\frac{t_m}{2}+\dots+\frac{t_1}{2^m},\frac{b_1}{2}+\dots+\frac{b_m}{2^m}\bigg) \ : \ t_1,\dots, t_m \in\{0,1\}\right\}, $$ where $b_k=t_k\oplus a_{k}t_{m}\oplus\sigma_m$ for $k\in\{1,\dots,m-1\}$ and $b_m=t_m\oplus \sigma_m$. The operation $\oplus$ denotes addition modulo 2. 

\item We also consider symmetrized versions of shifted digital nets, which we define as follows:
$$ \widetilde{\cP}_{\bsa}(\vecs):=\cP_{\bsa}(\vecs)\cup \cP_{\bsa}(\vecs^*)=\cP_{\bsa}(\vecs)\cup \{(x,1-2^{-m}-y):(x,y)\in \cP_{\bsa}(\vecs)\},$$
where $\vecs^*=(\sigma_1\oplus 1,\dots,\sigma_m\oplus 1)^{\top}$. 

\item Finally we introduce the class of digital $(0,m,2)$-nets $\cP_{\bsc}$ which are generated by $C_1$ as above and matrices $C_2$ of the form
\begin{equation*}
 C_2=
\begin{pmatrix}
1 & c_1 & c_1 & \cdots & c_1 & c_1 & c_1 \\
0 & 1 & c_2 & \cdots & c_2 & c_2 & c_2 \\
0 & 0 & 1 & \cdots & c_3 & c_3 & c_3 \\
\vdots & \vdots & \vdots & \ddots & \vdots & \vdots & \vdots & \\
0 & 0 & 0 & \cdots &  1 & c_{m-2} & c_{m-2} \\
0 & 0 & 0 & \cdots &  0 & 1 & c_{m-1} \\
0 & 0 & 0 & \cdots &  0 & 0 & 1 \\
\end{pmatrix},
\end{equation*}
where $\bsc=(c_1,\dots, c_{m-1})^{\top} \in \ZZ_2^{m-1}$. If $\bsc=(1,\dots,1)^{\top}$, the corresponding digital net $\cP_{\bsone}$ is sometimes called upper-$\bsone$-net.
\end{itemize}

\begin{thm} \label{2dnets}
We have
\begin{align*}
(L_{2,2^m}^{\mathrm{extr}}(\cP_{\bsa}(\vecs)))^2=\frac{m}{64}+\frac{1}{72}-\frac{1}{9\cdot 4^{m+2}}+\frac{1}{192}\sum_{k=1}^{m-2}a_k(2-2^{2k-2m+2})
\end{align*}
and
\begin{align*}
(L_{2,2^{m+1}}^{\mathrm{extr}}(\widetilde{\cP}_{\bsa}(\vecs))^2=\frac{m+1}{24}-\frac{5}{9\cdot 2^{4^{m+1}}}-\frac{1}{2^{2m+3}}\sum_{k=1}^{m-1}a_k 2^{2k}.
\end{align*}
For the upper-$\bsone$-net we have
\begin{align*}
(L_{2,2^m}^{\mathrm{extr}}(\cP_{\bsone}))^2=\frac{m}{64}+\frac{1}{72}-\frac{1}{9\cdot 4^{m+2}}.
\end{align*}
 For general $p\in (1,\infty)$ we have $$L_{p,N}^{\mathrm{extr}}(\cP) \lesssim_p \sqrt{\log N},$$ where $\cP\in \{\cP_{\bsa}(\vecs),\cP_{\bsc}\}$ and $N=2^m$ or $\cP=\widetilde{\cP}_{\bsa}(\vecs)$ and $N=2^{m+1}$. Hence, all these nets achieve the optimal order of extreme $L_p$ discrepancy, which is $\sqrt{\log{N}}$, according to Theorem~\ref{thm:l1}.
\end{thm}

\begin{proof}
Since the Haar coefficients of the respective discrepancy functions have already been computed, there is not much left to do. Just add the expressions given in~\cite[Lemmas~8-13]{Kritz} to obtain the result for $(L_{2,2^m}^{{\rm extr}}(\cP_{\bsa}(\vecs)))^2$. Caution: before applying the results from \cite[Lemmas~8-13]{Kritz} here, they have to be multiplied with $2^{2m}$ since in~\cite{Kritz} a normalized version of the discrepancy function is considered. 
  
Adding the results on the sums over $\mathcal{J}_8$ to $\mathcal{J}_{13}$ as stated in \cite[Lemma 14]{Kritz} (again multiplied with $2^{2(m+1)}$) yields the formula for $(L_{2,2^{m+1}}^{{\rm extr}}(\widetilde{\cP}_{\bsa}(\vecs))^2$. 

The Haar coefficients of the discrepancy function of the upper-$\bsone$-net have not been published so far; therefore we state the relevant results:
  \begin{itemize}
      \item Let $\mathcal{J}_1=\{(j_1,0):0\leq j_1 \leq n-3\}$. Then $$ \sum_{\bsj\in\mathcal{J}_1}2^{|\bsj|}\sum_{\bsm\in\DD_{\bsj}}|\mu_{\bsj\bsm}|^2=0. $$
      \item Let $\mathcal{J}_2=\{(j_1,j_2): j_1\geq 0, j_2 \geq 1, j_1+j_2\leq n-3\}$. Then $$ \sum_{\bsj\in\mathcal{J}_2}2^{|\bsj|}\sum_{\bsm\in\DD_{\bsj}}|\mu_{\bsj\bsm}|^2=\frac13 4^{-2n-5}\left(3n\cdot 4^n -5\cdot 2^{2n+1}+64\right). $$
      \item Let $\mathcal{J}_3=\{(j_1,j_2): j_1\geq 0, j_2 \geq 1, j_1+j_2= n-2\}$. Then $$ \sum_{\bsj\in\mathcal{J}_3}2^{|\bsj|}\sum_{\bsm\in\DD_{\bsj}}|\mu_{\bsj\bsm}|^2=\frac19 4^{-2n-5}\left(21\cdot n 4^n-2(5\cdot 4^n+256)\right). $$
       \item Let $\mathcal{J}_4=\{(n-2,0)\}$. Then $$ \sum_{\bsj\in\mathcal{J}_4}2^{|\bsj|}\sum_{\bsm\in\DD_{\bsj}}|\mu_{\bsj\bsm}|^2=\frac13 4^{2n-4} (4^n+32). $$
         \item Let $\mathcal{J}_5=\{(j_1,j_2): j_1\geq 0, j_2 \geq 1, j_1+j_2= n-1\}$. Then $$ \sum_{\bsj\in\mathcal{J}_5}2^{|\bsj|}\sum_{\bsm\in\DD_{\bsj}}|\mu_{\bsj\bsm}|^2=\frac{1}{27} 2^{-4n-7}\left(3n(5\cdot 4^n+32)-7\cdot 4^n -128\right). $$
      \item Let $\mathcal{J}_6=\{(n-1,0)\}$. Then $$ \sum_{\bsj\in\mathcal{J}_6}2^{|\bsj|}\sum_{\bsm\in\DD_{\bsj}}|\mu_{\bsj\bsm}|^2=\frac13 4^{-2n-3} (4^n+8). $$
         \item Let $\mathcal{J}_7=\{(j_1,j_2): j_1\geq n \text{\, or \,} j_2 \geq n\}$. Then $$ \sum_{\bsj\in\mathcal{J}_7}2^{|\bsj|}\sum_{\bsm\in\DD_{\bsj}}|\mu_{\bsj\bsm}|^2=\frac19 4^{-4n-4} (4\cdot 2^{2n+1}-1). $$
        \item Let $\mathcal{J}_8=\{(j_1,j_2): j_1+j_2\geq n \text{\, and \,} 1\leq j_1,j_2 \leq n-1\}$. Then $$ \sum_{\bsj\in\mathcal{J}_8}2^{|\bsj|}\sum_{\bsm\in\DD_{\bsj}}|\mu_{\bsj\bsm}|^2=\frac{1}{27} 4^{-2n-2}-\frac{1}{27}4^{-n-2}-\frac19 n 4^{-2n-1}+\frac59 n 4^{-n-3}. $$
  \end{itemize}
  Since $\mathcal{J}_1,\dots,\mathcal{J}_8$ form a partition of $\NN_0^2$, we have $$ (L_{2,2^m}^{\mathrm{extr}}(\cP_{\bsone}))^2=\sum_{i=1}^8 \sum_{\bsj\in\mathcal{J}_i}2^{|\bsj|}\sum_{\bsm\in\DD_{\bsj}}|\mu_{\bsj\bsm}|^2.$$ Inserting the above expressions for the single sub-sums yields the desired result.
  
 The claim on the extreme $L_p$ discrepancy follows from the mentioned results on the relevant Haar coefficients and the second part of Proposition~\ref{prop2}. The relevant Haar coefficients of the local discrepancy of the nets $\cP_{\bsc}$ can be found in~\cite[Lemma 3.2]{KP2019}.
\end{proof}

\begin{rem}\rm
We discuss the results in Theorem~\ref{2dnets}.
\begin{itemize}
    \item We remark that the digital shift $\vecs$ has no effect at all on the extreme $L_2$ discrepancy of $\cP_{\bsa}(\vecs)$. It would be interesting to know if this is a general rule which holds for all digital nets. Note that for the star $L_2$ discrepancy the shift is often crucial and can even reduce the order in $N$ for certain nets (see, e.g., \cite{KP06}). However, it seems that the only Haar coefficients that ``see'' the shift are those for $\bsj\in\NN_{-1}^2\setminus \NN_0^2$.
    \item Note that the Hammersley point set with $2^m$ elements has exactly the same extreme $L_2$ discrepancy as the upper-$\bsone$-net with the same number of points (for the result on the Hammersley point set see~\cite[Theorem~8]{HKP20} or choose $\bsa=(0,\dots,0)$ in the present Theorem~\ref{2dnets}). However, the Hammersley point set has a much higher star $L_2$ discrepancy, which is of order $\log{N}$, whereas the $L_2$ discrepancy of $\cP_{\bsone}$ is of optimal order $\sqrt{\log{N}}$ (see~\cite{KP2019}). The large $L_2$ discrepancy of the Hammersley point set is caused by the Haar coefficient $\mu_{(-1,-1),(0,0)}$.
    \item For the symmetrized nets $\widetilde{\cP}_{\bsa}(\vecs)$ we have $L_{2,2^{m+1}}^{{\rm star}}(\widetilde{\cP}_{\bsa}(\vecs))=\frac{m}{24}+\mathcal{O}(1)$ as well as $L_{2,2^{m+1}}^{\mathrm{extr}}(\widetilde{\cP}_{\bsa}(\vecs))=\frac{m}{24}+\mathcal{O}(1)$; hence the difference between the two $L_2$ discrepancies is not significant for these point sets. That is due to the fact that for the nets $\cP_{\bsa}(\vecs)$ the symmetrization has the effect to reduce the contribution of the Haar coefficients for $\bsj\in\NN_{-1}^2\setminus \NN_0^2$ to the $L_2$ discrepancy to order $\mathcal{O}(1)$ already.
    \item The result on $\cP_{\bsa}(\vecs)$ demonstrates that the Hammersley point set is not the digital $(0,m,2)$-net with largest extreme $L_2$ discrepancy. If we choose $\bsa=(0,\dots,0)$ (in which case $\cP_{\bsa}(\vecs)$ is the shifted Hammersley point set), then $(L_{2,2^m}^{{\rm extr}}(\cP_{\bsa}(\vecs)))^2=\frac{m}{64}+\mathcal{O}(1)$, while for $\bsa=(1,\dots,1)$ we only get $(L_{2,2^m}^{{\rm extr}}(\cP_{\bsa}(\vecs)))^2=\frac{m}{48}+\mathcal{O}(1)$.
    \item Open problem: Find a digital $(0,m,2)$-net which does not achieve the optimal order of extreme $L_2$ or $L_p$ discrepancy or prove that there is no such net. 
\end{itemize}
\end{rem}

\section{Periodic $L_p$ discrepancy}\label{sec:perdisc}
 
Now we turn our attention to the periodic $L_p$ discrepancy. For $p=2$ the lower bound of order of magnitude $(\log N)^{(d-1)/2}$ for $N$-element point sets in $[0,1)^d$ is well established (see~\eqref{lbd:Lper}). However, combining \eqref{rel:exper} and Corollary~\ref{co0} and Theorem~\ref{thm:l1}, respectively, this lower bound can be even obtained for the periodic $L_p$ discrepancy for general $p>1$. 

\begin{cor}\label{cor:lowper}
Let $p >1$. For every $N$-element point set $\cP$ in $[0,1)^d$ we have $$L_{p,N}^{{\rm per}}(\cP) \gtrsim_{p,d} (1+\log N)^{\frac{d-1}{2}}.$$ For $d=2$ the result even holds true for $p=1$.
\end{cor}

We show that this lower bound is best possible in the order of magnitude in $N$ for fixed dimension~$d$ and $p \le 2$. In the following we will consider the periodic $L_2$ discrepancy only. Unfortunately, we cannot apply the simple argument used for the extreme $L_2$ discrepancy here, since in general we do not know whether the periodic $L_2$ discrepancy is dominated by the star $L_2$ discrepancy. (We remark that Lev~\cite{Lev} showed that for certain symmetrized point sets $\cP^{{\rm sym}}$ it is indeed true that $L_{2,N}^{{\rm per}}(\cP^{{\rm sym}}) \lesssim L_{2,N}(\cP^{{\rm sym}})$. However, it is not known if any of the point sets of optimal star $L_2$ discrepancy satisfies these symmetry properties.)\\

It is a known fact that  the periodic $L_2$ discrepancy can be expressed in terms of exponential sums. For $\cP=\{\bsx_0,\bsx_1,\ldots,\bsx_{N-1}\}$ in $[0,1)^d$ we have 
\begin{equation}\label{pr_dia}
(L_{2,N}^{{\rm per}}(\cP))^2=\frac{1}{3^d} \sum_{\bsk \in \ZZ^d\setminus\{\bszero\}} \frac{1}{r(\bsk)^2} \left| \sum_{n=0}^{N-1} \exp(2 \pi \icomp \bsk \cdot \bsx_n)\right|^2,
\end{equation}
 where $\icomp=\sqrt{-1}$ and where for $\bsk=(k_1,\ldots,k_d)\in \ZZ^d$ we set 
\begin{equation*}
r(\bsk)=\prod_{j=1}^d r(k_j) \ \ \ \mbox{ and } \ \ r(k_j)=\left\{ 
\begin{array}{ll}
1 & \mbox{ if $k_j=0$},\\
\frac{2 \pi |k_j|}{\sqrt{6}} & \mbox{ if $k_j\not=0$.} 
\end{array}\right.
\end{equation*}
For a proof of this relation see \cite[Theorem~1]{Lev} or \cite[p.~390]{HOe}. Formula \eqref{pr_dia} shows that the periodic $L_2$ discrepancy, normalized by $N$, is -- up to a multiplicative factor -- exactly the diaphony which is a well-known measure for the irregularity of distribution of point sets and which was introduced by Zinterhof~\cite{zint} in the year 1976 (see also \cite{DT} and the forthcoming Remark~\ref{re:dia}).

\begin{rem}\label{re:lev}\rm 
We recall that Lev~\cite{lev99} studied a slightly more general notion of diaphony which involve certain weights and which he called generalized diaphony. For a certain choice of weights this generalized diaphony coincides with the ``classical Zinterhof diaphony'' and therefore -- up to multiplicative factors -- with the periodic discrepancy studied in the present chapter. He was able to determine the exact order of magnitude in $N$ for fixed dimension $d$ (see \cite[Main Theorem]{lev99}). While the lower bounds hold for any of the involved weights and therefore also for the periodic $L_2$ discrepancy considered in the present paper, the upper bounds are achieved only for certain weights which do not comprise the setting considered here. Just as a side note, the construction is based on Frolov's construction of lattices from \cite{Fro}. 
\end{rem}
 
Now we re-write \eqref{pr_dia} further and get this way 
\begin{eqnarray*}
(L_{2,N}^{{\rm per}}(\cP))^2 & = & \frac{1}{3^d} \sum_{\bsk \in \ZZ^d\setminus\{\bszero\}} \frac{1}{r(\bsk)^2} \left| \sum_{n=0}^{N-1} \exp(2 \pi \icomp \bsk \cdot \bsx_n)\right|^2\\
& = &  -\frac{N^2}{3^d}  + \frac{1}{3^d} \sum_{n,p=0}^{N-1}  K_d(\bsx_n,\bsx_p),
\end{eqnarray*}
where $$K_d(\bsx,\bsy):=   \sum_{\bsk \in \ZZ^d} \frac{1}{r(\bsk)^2} \exp(2 \pi \icomp \bsk \cdot (\bsx-\bsy)).$$

Our aim is now to find a Walsh-series representation of the periodic $L_2$ discrepancy.

\paragraph{Walsh functions in base $b$.} For $k \in \NN_0$ with $b$-adic representation
\[
   k = \kappa_{a-1} b^{a-1} + \cdots + \kappa_1 b + \kappa_0,
\]
with $\kappa_i \in \ZZ_b$, we define the ($b$-adic) Walsh function $\wal_{k}:[0,1) \rightarrow \{z \in \CC \ : \ |z|=1\}$ by
\[
  \wal_{k}(x) := {\rm e}^{2 \pi \icomp (\xi_1 \kappa_0 + \cdots + \xi_a \kappa_{a-1})/b},
\]
for $x \in [0,1)$ with $b$-adic representation $x = \frac{\xi_1}{b}+\frac{\xi_2}{b^2}+\cdots$, with $\xi_i \in \ZZ_b$, (unique in the sense that infinitely many of the $\xi_i$ must be different from $b-1$).

For dimension $d \geq 2$, $\bsx=(x_1, \ldots, x_d) \in [0,1)^d$ and $\bsk=(k_1, \ldots, k_d) \in \NN_0^d$ we define $\wal_{\bsk} : [0,1)^d \rightarrow \{z \in \CC \ : \ |z|=1\}$
by
\[
   \wal_{\bsk}(\bsx) := \prod_{j=1}^d \wal_{k_j}(x_j).
\]

Information about basic properties of Walsh functions, especially in the context of digital nets, can be found in \cite[Appendix~A]{DP10}.

\paragraph{Walsh series expansion.} Now we expand $K_d(\bsx,\bsy)$ into a Walsh series. Fix a prime number $b$. We have $$K_d(\bsx,\bsy)=\sum_{\bsk,\bsell \in \NN_0^d} \rho_b(\bsk,\bsell) \wal_{\bsk}(\bsx) \overline{\wal_{\bsell}(\bsy)},$$ where $$\rho_b(\bsk,\bsell)=\int_{[0,1]^{2d}} K_d(\bsx,\bsy) \overline{\wal_{\bsk}(\bsx)} \wal_{\bsell}(\bsy) \rd \bsx \rd \bsy.$$

Due to the multiplicative structure of $K_d$ and of the multi-dimensional Walsh functions, for $\bsk=(k_1,\ldots,k_d)$ and $\bsell=(\ell_1,\ldots,\ell_d)$ in $\NN_0^d$ we have $$\rho_b(\bsk,\bsell)=\prod_{j=1}^d \rho_b(k_j,\ell_j),$$
where, for $d=1$ and $k,\ell \in \NN_0$, 
\begin{eqnarray*}
\rho_b(k,\ell) & = & \int_{[0,1]^{2}} K_1(x,y) \overline{\wal_{k_j}(x)} \wal_{\ell_j}(y) \rd x \rd y\\
& = & \int_0^1 \int_0^1 \sum_{h=-\infty}^{\infty} \frac{1}{r(h)^2} \exp(2\pi \icomp h(x-y)) \overline{\wal_k(x)} \wal_{\ell}(y) \rd x \rd y\\
& = & \sum_{h=-\infty}^{\infty} \frac{1}{r(h)^2}  \overline{\left(\int_0^1 \exp(-2 \pi \icomp hx) \wal_k(x)  \rd x \right)}\left(\int_0^1 \exp(-2 \pi \icomp h y) \wal_\ell(y)  \rd y \right).
\end{eqnarray*}  
Put $$\beta_{h,k}:=\int_0^1 \exp(-2 \pi \icomp hx) \wal_k(x)  \rd x.$$ Note that $\beta_{0,0}=1$, $\beta_{0,k}=0$ for $k \in \NN$ and $\beta_{h,0}=0$ for $h \in \ZZ\setminus\{0\}$. Then we can write $$\rho_b(k,\ell) =  \sum_{h=-\infty}^{\infty} \frac{\overline{\beta}_{h,k} \beta_{h,\ell}}{r(h)^2}$$ and, in particular,  $\rho_b(0,0)=1$ and $\rho_b(0,\ell)=\rho_b(k,0)=0$ for $k,\ell \in \NN$. For $k,\ell \in \NN$ we have $$\rho_b(k,\ell) =  \sum_{h=-\infty\atop h \not=0}^{\infty} \frac{\overline{\beta}_{h,k} \beta_{h,\ell}}{r(h)^2}=\frac{3}{2 \pi^2} \sum_{h=-\infty\atop h \not=0}^{\infty} \frac{\overline{\beta}_{h,k} \beta_{h,\ell}}{h^2}.$$

Using the Walsh series expansion of $K_d(\cdot,\cdot)$ we can write the periodic $L_2$ discrepancy as 
\begin{eqnarray}\label{fo:perL2wal}
(L_{2,N}^{{\rm per}}(\cP))^2  = -\frac{N^2}{3^d}  + \frac{1}{3^d} \sum_{n,p=0}^{N-1} \sum_{\bsk,\bsell \in \NN_0^d} \rho_b(\bsk,\bsell) \wal_{\bsk}(\bsx_n) \overline{\wal_{\bsell}(\bsx_p)}. 
\end{eqnarray}

Unfortunately, the coefficients $\rho_b(k,\ell)$ for $k \not=\ell$ (the ``non-diagonal terms'') are rather difficult to handle. However, one can get rid of them if one considers the root mean square of digitally shifted point sets with respect to all digital shifts of depth $m$. This follows immediately from the following lemma.

\begin{lem}\label{shift_wal}
Let $x_1, x_2 \in [0,1)$ and let $z_1, z_2 \in [0,1)$ be the points obtained after applying an i.i.d. random digital shift of depth $m$ to $x_1$ and $x_2$. Then we have 
\begin{equation*}
\EE[\wal_k(z_{1}) \overline{\wal_l(z_{2})}]  = \left\{ \begin{array}{ll}
\wal_k(x_{1})\overline{\wal_k(x_{2})} & \mbox{if } 0 \le k = l <b^m, \\ 0 & \mbox{otherwise}. \end{array} \right.
\end{equation*}
\end{lem}

\begin{proof}
A proof of this lemma in the case $b=2$ is given in \cite[Lemma~3]{DP05}. For general prime $b$ the proof follows the same arguments.
\end{proof}

Let $\cP=\{\bsx_0,\bsx_1,\ldots,\bsx_{N-1}\}$ be a point set in $[0,1)^d$ and let $\widetilde{\cP}=\{\bsz_0,\bsz_1,\ldots,\bsz_{N-1}\}$ be the digitally shifted (of depth $m$) version thereof. Then, applying Lemma~\ref{shift_wal} to \eqref{fo:perL2wal}, we obtain 
\begin{eqnarray}\label{fe:experL2wal}
\EE[(L_{2,N}^{{\rm per}}(\widetilde{\cP}))^2] & = &  -\frac{N^2}{3^d}  + \frac{1}{3^d} \sum_{n,p=0}^{N-1} \sum_{\bsk,\bsell \in \NN_0^d} \rho_b(\bsk,\bsell) \EE[\wal_{\bsk}(\bsz_n) \overline{\wal_{\bsell}(\bsz_p)}]\nonumber\\
& = &  -\frac{N^2}{3^d}  + \frac{1}{3^d} \sum_{n,p=0}^{N-1} \sum_{\bsk \in \{0,1,\ldots,b^m-1\}^d} \rho_b(\bsk,\bsk) \wal_{\bsk}(\bsx_n) \overline{\wal_{\bsk}(\bsx_p)}\nonumber\\
& = &  -\frac{N^2}{3^d}  + \frac{1}{3^d} \sum_{\bsk \in \{0,1,\ldots,b^m-1\}^d} \rho_b(\bsk,\bsk) \left|\sum_{n=0}^{N-1} \wal_{\bsk}(\bsx_n)\right|^2.
\end{eqnarray}

In this formula only the ``diagonal terms'' $\rho_b(k,k)$ appear. These terms can be explicitly computed. In the following we simplify the notation and write $\rho_b(k):=\rho_b(k,k)$ for $k \in \NN_0$.

\begin{lem}\label{le:rhok}
For $k \in \NN$ we have 
$$\rho_b(k) = \left\{
\begin{array}{ll}
\frac{3}{b^{2a}} \left(-\frac{1}{3} + \frac{1}{2 \sin^2(\kappa_{a-1}\pi/b)} \right)& \mbox{ if $k=\kappa_{a-1} b^{a-1}$ with $a\in \NN$ and $\kappa_{a-1}\in \ZZ_b^*$},\\[1em]
\frac{3}{b^{2a}} \left(-\frac{1}{3} + \frac{1}{\sin^2(\kappa_{a-1}\pi/b)} \right) & \mbox{ if $k=\kappa_{a-1} b^{a-1}+k'$ with $a\in \NN$, $\kappa_{a-1} \in \ZZ_b^*$}\\
& \mbox{ and $1 \le k' < b^{a-1}$},
\end{array}\right.$$
where $\ZZ_b^*:=\ZZ_b \setminus \{0\}$.
\end{lem}

\begin{proof}
For $k \in \NN$ we have
\begin{eqnarray*}
\rho_b(k) & = & \int_0^1 \int_0^1 \sum_{h=-\infty}^{\infty} \frac{1}{r(h)^2} \exp(2\pi \icomp h(x-y)) \overline{\wal_k(x)} \wal_{k}(y) \rd x \rd y\\
& = & \int_0^1 \int_0^1 \left(1+\frac{3}{2 \pi^2} \sum_{h=-\infty\atop h \not=0 }^{\infty} \frac{1}{h^2} \exp(2\pi \icomp h(x-y))\right) \overline{\wal_k(x)} \wal_{k}(y) \rd x \rd y.
\end{eqnarray*}

The second Bernoulli polynomial is given by $B_2(x)=x^2-x+\frac{1}{6}$  and it is well known, that this polynomial has the Fourier expansion $$B_2(x)=\frac{1}{2 \pi^2} \sum_{h=-\infty \atop h \not=0}^{\infty} \frac{\exp(2 \pi \icomp h x)}{h^2}\quad\quad \mbox{for $x \in (0,1)$.}$$ From this we see that 
\begin{eqnarray*}
\rho_b(k)& = & \int_0^1 \int_0^1 \left(1+3 B_2(|x-y|)\right) \overline{\wal_k(x)} \wal_{k}(y) \rd x \rd y\\
& = & \int_0^1 \int_0^1 \left(\frac{3}{2}+3 |x-y|^2-3 |x-y|\right) \overline{\wal_k(x)} \wal_{k}(y) \rd x \rd y\\
& = & 3  \int_0^1 \int_0^1 |x-y|^2\ \overline{\wal_k(x)} \wal_{k}(y) \rd x \rd y - 3 \int_0^1 \int_0^1 |x-y| \  \overline{\wal_k(x)} \wal_{k}(y) \rd x \rd y.
\end{eqnarray*}

It follows from \cite[Appendix~A]{DP05} that for $k = \kappa_{a-1} b^{a-1} + \cdots + \kappa_1 b + \kappa_0$, with $\kappa_{a-1} \in \ZZ_b^*$ we have $$\int_0^1 \int_0^1 |x-y| \  \overline{\wal_k(x)} \wal_{k}(y) \rd x \rd y = \frac{1}{b^{2a}} \left(\frac{1}{3} - \frac{1}{\sin^2(\kappa_{a-1}\pi/b)} \right).$$
Inserting this we obtain further
\begin{eqnarray}\label{le:furt}
\rho_b(k)& = & -6 \int_0^1 x \, \overline{\wal_k(x)} \rd x \int_0^1 y\,  \wal_k(y) \rd y -\frac{3}{b^{2a}} \left(\frac{1}{3} - \frac{1}{\sin^2(\kappa_{a-1}\pi/b)} \right)\nonumber\\
& = &  -6 \left(\int_0^1 x \, \overline{\wal_k(x)} \rd x\right)^2 -\frac{1}{b^{2a}} \left(1 - \frac{3}{\sin^2(\kappa_{a-1}\pi/b)} \right)
\end{eqnarray}  
where $k = \kappa_{a-1} b^{a-1} + \cdots + \kappa_1 b + \kappa_0$, with $\kappa_{a-1} \in \ZZ_b^*$.
We have 
$$ \int_0^1 x \, \overline{\wal_k(x)} \rd x =\left\{ 
\begin{array}{ll}
\frac{1}{b^a ({\rm e}^{-2 \pi \icomp \kappa_{a-1}/b} -1)} & \mbox{ if $k=\kappa_{a-1} b^{a-1}$ with $\kappa_{a-1} \in \ZZ_b^*$},\\
0 & \mbox{ otherwise}.
\end{array}\right.$$
Hence
$$\left( \int_0^1 x \, \overline{\wal_k(x)} \rd x\right)^2 =\left\{ 
\begin{array}{ll}
\frac{1}{4 b^{2a} \sin^2(\kappa_{a-1} \pi/b)} & \mbox{ if $k=\kappa_{a-1} b^{a-1}$ with $\kappa_{a-1} \in \ZZ_b^*$},\\
0 & \mbox{ otherwise}.
\end{array}\right.$$ 
Inserting this result into \eqref{le:furt} gives the desired result.
\end{proof}

\paragraph{Periodic discrepancy of digitally shifted digital nets.}

Now assume that we are given a digital $(t,m,d)$-net over $\ZZ_b$, where $b$ is the same prime number as used in the definition of the Walsh functions. Denote the $m \times m$ generating matrices of the digital net by $C_1,C_2,\ldots,C_d$. The dual net is defined as $$\calD=\calD(C_1,\ldots,C_d)=\{(k_1,\ldots,k_d) \in \NN_0^d \ : \ C_1^\top \vec{k}_1+\cdots +C_d^\top \vec{k}_d = \vec{0}\}.$$ Here, for $k \in \NN_0$ with $b$-adic expansion $k=\kappa_0+\kappa_1 b + \kappa_2 b^2+\cdots$ with digits $\kappa_i \in \ZZ_b$, $i \in \NN_0$, we write $\vec{k}=(\kappa_0,\kappa_1,\ldots,\kappa_{m-1})^\top$. It is well known (see, e.g., \cite[Lemma~4.75]{DP10}), that for a digital net $\{\bsx_0,\bsx_1,\ldots,\bsx_{b^m -1}\}$ we have $$\sum_{n=0}^{b^m-1} \wal_{\bsk}(\bsx_n) =\left\{ 
\begin{array}{ll}
b^m & \mbox{ if $\bsk \in \calD$,}\\
0 & \mbox{ if $\bsk \not\in \calD$.}
\end{array}
\right.$$
From \eqref{fe:experL2wal} we therefore obtain
\begin{eqnarray*}
\EE[(L_{2,b^m}^{{\rm per}}(\widetilde{\cP}))^2] = \frac{b^{2m}}{3^d}  \sum_{\bsk \in \calD^*\cap \{0,1,\ldots,b^m-1\}^d} \rho_b(\bsk),
\end{eqnarray*}
where $\calD^*:=\calD\setminus\{\bszero\}$. We re-write the above expression further since this will be useful for the following considerations. We can write
\begin{eqnarray}\label{fo:l2perB}
\EE[(L_{2,b^m}^{{\rm per}}(\widetilde{\cP}))^2] = \frac{b^{2m}}{3^d} \sum_{\emptyset \not=\uu \subseteq [d]} \calB(\uu),
\end{eqnarray}
where 
\begin{equation*}
\calB(\uu):= \sum_{\bsk_{\uu} \{1,\ldots,b^m-1\}^{|\uu|} \atop \sum_{j \in \uu} C_j^\top \vec{k}_j=\vec{0}} \prod_{j \in \uu} \rho_b(k_j).
\end{equation*}

\begin{thm}\label{thm12}
Let $\cP$ be a digital $(t,m,d)$-net over $\ZZ_b$ and let $\widetilde{\cP}$ be a digitally shifted (of depth $m$) version of this net. Then the mean-square periodic $L_2$ discrepancy over all digital shifts of depth $m$ is bounded as
\begin{eqnarray*}
\EE[(L_{2,b^m}^{{\rm per}}(\widetilde{\cP}))^2] \le  b^{2t}  (m-t)^{d-1} \left(\frac{1+b^2}{3}\right)^d.
\end{eqnarray*}
\end{thm}

For the proof of Theorem~\ref{thm12} we need a further lemma.

\begin{lem}\label{B_bound}
Let $C_1,\ldots ,C_d$ be the generating matrices of a digital $(t,m,d)$-net over $\ZZ_b$. Further define $\calB$ as above. Then for any $\uu \subseteq [d]$, $\uu \not=\emptyset$, we have $$\calB(\uu) \le \frac{b^{2t}}{b^{2m}} (m-t)^{|\uu|-1}  b^{2|\uu|}.$$
\end{lem}

\begin{proof} 
We use the method from \cite[Proof of Lemma~7]{CDP06}. To simplify the notation we show the result only for $\uu = [d]$. The other cases follow by the same arguments. We have, for $k_j=\kappa_{j,0}+\kappa_{j,1} b + \cdots +\kappa_{j,a_j-1} b^{a_j-1}$ and $k_{j,a_j-1}\in \ZZ_b^*$ for $j \in [d]$,  
\begin{eqnarray*}
\calB([d]) &=& \underbrace{\sum_{k_1=1}^{b^{m}-1}\ldots \sum_{k_d=1}^{b^{m}-1}}_{C_1^{\top}\vec{k}_1+\cdots +C_d^{\top} \vec{k}_d=\vec{0}} \prod_{j=1}^d  \rho_b(k_j)  \nonumber \\ 
& = & \sum_{a_1,\ldots
,a_d=1}^{m} \prod_{j=1}^d \frac{3}{b^{2a_j}} \underbrace{\sum_{k_1=b^{a_1-1}}^{b^{a_1}-1}\ldots \sum_{k_d=b^{a_d-1}}^{b^{a_d}-1}}_{C_1^{\top}\vec{k}_1+\cdots +C_d^{\top}
\vec{k}_d=\vec{0}}  \prod_{j=1}^d  \left(-\frac{1}{3}+ \frac{1}{z(k_j)\sin^2(\kappa_{a_j-1} \pi/b)}\right),
\end{eqnarray*}
where $z(k_j)=2$, if $k_j' =0$ and $z(k_j)=1$ otherwise (remember that $k_j'=\kappa_{j,0}+\kappa_{j,1} b + \cdots +\kappa_{j,a_j-2} b^{a_j-2}$). Hence
\begin{equation}\label{cond1}
\calB([d])  \le  \sum_{a_1,\ldots
,a_d=1}^{m} \prod_{j=1}^d \frac{3}{b^{2a_j}} \underbrace{\sum_{k_1=b^{a_1-1}}^{b^{a_1}-1}\ldots \sum_{k_d=b^{a_d-1}}^{b^{a_d}-1}}_{C_1^{\top}\vec{k}_1+\cdots +C_d^{\top}
\vec{k}_d=\vec{0}}  \prod_{j=1}^d  \left(-\frac{1}{3}+ \frac{1}{\sin^2(\kappa_{a_j-1} \pi/b)}\right).
\end{equation}

For $j \in [d]$ and $i \in [m]$ let $\vec{c}_{j,i}^{\;\top}$ denote the $i$-th row vector of the matrix $C_j$. For $b^{a_j-1}\le k_j \le b^{a_j}-1$, the $b$-adic digit expansion of $k_j$ is of the form $$k_j=\kappa_{j,0}+\kappa_{j,1} b+ \cdots + \kappa_{j,a_j-2} b^{a_j-2}+\kappa_{j,a_j-1}b^{a_j-1}$$ with $\kappa_{j,a_j-1}\in \ZZ_b^*$.  Hence the condition in  sum \eqref{cond1} can be written as  
\begin{eqnarray}\label{lgs1}
\vec{c}_{1,1} \kappa_{1,0}+\cdots +\vec{c}_{1,a_1-1}
\kappa_{1,a_1-2}+\vec{c}_{1,a_1}\kappa_{1,a_1-1}+&&\nonumber\\
\vec{c}_{2,1} \kappa_{2,0}+\cdots +\vec{c}_{2,a_2-1} \kappa_{2,a_2-2}+\vec{c}_{2,a_2}\kappa_{2,a_2-1}+&&\nonumber\\
\vdots && \\
\vec{c}_{s,1} \kappa_{d,0}+\cdots +\vec{c}_{d,a_d-1}
\kappa_{d,a_d-2}+\vec{c}_{d,a_d}\kappa_{d,a_d-1} \hspace{0.3cm}& = & \vec{0}.\nonumber
\end{eqnarray}

Since by the digital $(t,m,d)$-net property (see Definition \ref{def2}) the vectors $$\vec{c}_{1,1},\ldots ,\vec{c}_{1,a_1},\ldots ,\vec{c}_{d,1},\ldots
,\vec{c}_{d,a_d}$$ are linearly independent as long as
$a_1+\cdots +a_d \le m-t,$ we must have 
\begin{equation}\label{bed1}
a_1+\cdots +a_d \ge m-t+1.
\end{equation}
Let now $A$ denote the $m \times (a_1+\cdots +a_d-d)$ matrix with the 
column vectors given by $\vec{c}_{1,1},\ldots ,\vec{c}_{1,a_1-1},\ldots ,\vec{c}_{d,1},\ldots ,\vec{c}_{d,a_d-1}$, i.e., $$A:=(\vec{c}_{1,1},\ldots ,\vec{c}_{1,a_1-1},\ldots ,\vec{c}_{d,1},\ldots ,\vec{c}_{d,a_d-1}).$$ Further let $$\vec{f}_{\kappa_{1,a_1-1},\ldots,\kappa_{d,a_d-1}}:=-(\vec{c}_{1,a_1}\kappa_{1,a_1-1}+\cdots +\vec{c}_{d,a_d}\kappa_{d,a_d-1}) \in \ZZ_b^m$$ and
$$\vec{k}:=(\kappa_{1,0},\ldots ,\kappa_{1,a_1-2},\ldots ,\kappa_{d,0},\ldots
,\kappa_{d,a_d-2})^{\top} \in \ZZ_b^{a_1+\cdots+a_d -d}.$$
Then the linear equation system (\ref{lgs1}) can be written as 
\begin{equation}\label{lgs2}
A\vec{k} = \vec{f}_{\kappa_{1,a_1-1},\ldots,\kappa_{d,a_d-1}}
\end{equation}
and hence

\begin{eqnarray*}
\lefteqn{\underbrace{\sum_{k_1=b^{a_1-1}}^{b^{a_1}-1}\ldots \sum_{k_d=b^{a_d-1}}^{b^{a_d}-1}}_{C_1^{\top}\vec{k}_1+\cdots +C_d^{\top}
\vec{k}_d=\vec{0}}  \prod_{j=1}^d  \left(-\frac{1}{3}+ \frac{1}{\sin^2(\kappa_{a_j-1} \pi/b)}\right)}\\
&=& \sum_{\kappa_{1,a_1-1},\ldots,\kappa_{d,a_d-1}=1}^{b-1} \prod_{j=1}^d \left(-\frac{1}{3}+\frac{1}{\sin^2(\kappa_{j,a_j-1} \pi /b)}\right)   \sum_{\vec{k}\in \ZZ_b^{a_1+\cdots +a_d-d}\atop A \vec{k}=\vec{f}_{\kappa_{1,a_1-1},\ldots,\kappa_{d,a_d-1}}}1\\
&=& \sum_{\kappa_{1,a_1-1},\ldots,\kappa_{d,a_d-1}=1}^{b-1} \prod_{j=1}^d \left(-\frac{1}{3}+\frac{1}{\sin^2(\kappa_{j,a_j-1} \pi /b)}\right) \#\{\vec{k}\in \ZZ_b^{a_1+\cdots +a_d-d}\ :\  A \vec{k}=\vec{f}_{\kappa_{1,a_1-1},\ldots,\kappa_{d,a_d-1}}\}.
\end{eqnarray*}
By the definition of the matrix $A$ and since $C_1,\ldots ,C_d$ are
the generating matrices of a digital $(t,m,d)$-net over $\ZZ_b$ we
have
$${\rm rank}(A)= \left\{
\begin{array}{ll}
a_1+\cdots +a_d -d & \mbox{ if } a_1+\cdots +a_d-d \le m-t,\\ 
\ge m-t & \mbox{ else}.
\end{array}
\right.$$
Let $L$ denote the linear space of solutions of the homogeneous system
$A\vec{k}=\vec{0}$ and let ${\rm dim}(L)$ denote the dimension of
$L$. Then it follows that $${\rm dim}(L) = \left\{
\begin{array}{ll}
0 & \mbox{ if } a_1+\cdots +a_d-d\le m-t,\\ 
\le a_1+\cdots +a_d-d-m+t & \mbox{ else}.
\end{array}
\right.$$
Hence if $a_1+\cdots +a_d-d \le m-t$ we find that the system \eqref{lgs2}
has at most 1 solution and if $a_1+\cdots +a_d-d > m-t$ the system
\eqref{lgs2} has at most $b^{a_1+\cdots+a_d-d-m+t}$ solutions, i.e.,
\begin{eqnarray*}
\lefteqn{\underbrace{\sum_{k_1=b^{a_1-1}}^{b^{a_1}-1}\ldots \sum_{k_d=b^{a_d-1}}^{b^{a_d}-1}}_{C_1^{\top}\vec{k}_1+\cdots +C_d^{\top}
\vec{k}_d=\vec{0}}  \prod_{j=1}^d  \left(-\frac{1}{3}+ \frac{1}{\sin^2(\kappa_{a_j-1} \pi/b)}\right)}\\
& \le &  \sum_{\kappa_{1,a_1-1},\ldots,\kappa_{d,a_d-1}=1}^{b-1} \prod_{j=1}^d \left(-\frac{1}{3}+\frac{1}{\sin^2(\kappa_{j,a_j-1} \pi /b)}\right)\\
&& \hspace{2cm}\times \left\{
\begin{array}{ll}
1 & \mbox{ if } a_1+\cdots +a_d-d \le m-t,\\
b^{a_1+\cdots +a_d-d-m+t} & \mbox{ if } a_1+\cdots +a_d-d > m-t.
\end{array}\right.
\end{eqnarray*}
In \cite[Appendix C]{DP05a} it is shown that $\sum_{\kappa=1}^{b-1}\frac{1}{\sin^2(\kappa \pi/b)}=\frac{b^2-1}{3}$. Hence
\begin{eqnarray*}
\lefteqn{\underbrace{\sum_{k_1=b^{a_1-1}}^{b^{a_1}-1}\ldots \sum_{k_d=b^{a_d-1}}^{b^{a_d}-1}}_{C_1^{\top}\vec{k}_1+\cdots +C_d^{\top}
\vec{k}_d=\vec{0}}  \prod_{j=1}^d  \left(-\frac{1}{3}+ \frac{1}{\sin^2(\kappa_{a_j-1} \pi/b)}\right)}\\
& \le &  \left(\frac{b^2-b}{3}\right)^d  \times \left\{
\begin{array}{ll}
1 & \mbox{ if } a_1+\cdots +a_d \le m-t+d,\\
b^{a_1+\cdots +a_d-d-m+t} & \mbox{ if } a_1+\cdots +a_d > m-t+d.
\end{array}\right.
\end{eqnarray*}
 Therefore together with condition \eqref{bed1} we obtain
\begin{eqnarray*}
\calB([d])\le (b^2-b)^d \left(  \sum_{a_1,\ldots ,a_d=1 \atop m-t+1 \le a_1+\cdots +a_d \le m-t+d}^{m} \frac{1}{b^{2(a_1+\cdots +a_d)}} + \sum_{a_1,\ldots ,a_d=1 \atop  a_1+\cdots
+a_d > m-t+d}^{m} \frac{b^{a_1+\cdots
+a_d-d-m+t}}{b^{2(a_1+\cdots +a_d)}}\right).
\end{eqnarray*}

Now we have to estimate the two sums in the above expression. First we have
\begin{eqnarray*}
\Sigma_{1} &:=& \sum_{a_1,\ldots ,a_d=1 \atop  a_1+\cdots
+a_d > m-t+d}^{m} \frac{b^{a_1+\cdots
+a_d-d-m+t}}{b^{2(a_1+\cdots +a_d)}}\\
& = & \frac{b^t}{b^{m+d}} \sum_{l=m-t+d+1}^{d m}\frac{1}{b^l} \sum_{a_1,\ldots ,a_d=1\atop a_1+\cdots
+a_d=l}^{m}1\\
& = & \frac{b^t}{b^{m+d}} \sum_{l=m-t+1}^{d (m-1)}\frac{1}{b^{l+d}} \sum_{a_1,\ldots ,a_d=0\atop a_1+\cdots
+a_d=l}^{m-1}1\\
&  \le & \frac{b^t}{b^{m+2d}}
\sum_{l=m-t+1}^{\infty} {l+d-1 \choose d-1} \frac{1}{b^l},
\end{eqnarray*}
where we used the fact that for fixed $l$ the number of non-negative integer
solutions of $a_1+\cdots +a_d=l$ is given by ${l+d-1 \choose d-1}$. 

It follows from the binomial theorem (see, e.g., \cite[Lemma~6]{DP05}) that for $b>1$ and integers $d,t_0>0$ we have
\begin{equation}\label{est:binom}
\sum_{l=t_0}^{\infty} {l+d-1 \choose d-1} \frac{1}{b^l} \le b^{-t_0} {t_0+d-1 \choose d-1} \left(\frac{b-1}{b}\right)^{-d}. 
\end{equation}

Using \eqref{est:binom} we now obtain
\begin{eqnarray}\label{sum1_asym}
\Sigma_{1} & \le &  \frac{b^t}{b^{m+2d}} \frac{1}{b^{m-t+1}}
{m-t+d \choose d-1}\left(\frac{b-1}{b} \right)^{-d} \nonumber\\
& =&   \frac{b^{2t}}{b^{2m}} \frac{1}{b(b^2-b)^d}
{m-t+d \choose d-1}.
\end{eqnarray}
Finally, since $${m-t+d \choose d-1} = \frac{(m-t+2)(m-t+3)\cdots
(m-t+d)}{1 \cdot 2 \cdots (d-1)}\le (m-t+2)^{d-1},$$ we
obtain $$\Sigma_1 \le  \frac{b^{2t}}{b^{2m}} \frac{1}{b(b^2-b)^d} (m-t+2)^{d-1}.$$

Now we estimate $$\Sigma_{2}:=\sum_{a_1,\ldots ,a_d=1 \atop m-t+1 \le a_1+\cdots +a_d \le m-t+d}^{m} \frac{1}{b^{2(a_1+\cdots +a_d)}}. $$ If $m-t \geq d-1$ we proceed similarly to above and obtain
\begin{eqnarray}\label{sum2_asym}
\Sigma_{2} & = & \sum_{l=m-t-d+1}^{m-t} {l+d-1 \choose d-1} \frac{1}{b^{2(l+d)}}  \nonumber \\
& \le &   \frac{1}{b^{2(m-t+1)}}{ m-t \choose d-1} \left(\frac{b^2-1}{b^2}\right)^{-d} \nonumber \\
& = & \frac{b^{2d}}{b^2(b^2-1)^d} \frac{b^{2t}}{b^{2m}} {m-t \choose d-1}. 
\end{eqnarray}
If $m-t<d-1$, then we have 
\begin{eqnarray}\label{sum2_asym2}
\Sigma_2 &=&   \sum_{l=d}^{m-t+d} \frac{1}{b^{2l}} \sum_{a_1,\ldots,a_d=1 \atop a_1+\cdots+a_d=l}^m 1 \nonumber \\ 
&\le&   \sum_{l=0}^{m-t} \frac{1}{b^{2(l+d)}} \sum_{a_1,\ldots,a_d=0 \atop a_1+\cdots+a_d=l}^{m-1} 1 \nonumber \\ 
& \leq & \frac{1}{b^{2d}} \sum_{l=0}^\infty {l+d-1 \choose d-1} \frac{1}{b^{2l}} \nonumber \\
& = &  \frac{1}{b^{2d}}  \frac{b^{2d}}{(b^2-1)^d} \\
& \le &   \frac{b^{2d}}{b^2 (b^2-1)^d} \frac{b^{2t}}{b^{2m}},
\end{eqnarray}
where we used $b^{2-2d} \le b^{2t-2m}$ in the last step. Hence we have in any case
$$\Sigma_2 \le \frac{b^{2d}}{b^2 (b^2-1)^d} \frac{b^{2t}}{b^{2m}} \max\left(1, {m-t \choose d-1}\right).$$

Putting things together we finally obtain 
\begin{eqnarray*}
\calB([d])& \le & (b^2-b)^d \left(\frac{b^{2d}}{b^2 (b^2-1)^d} \frac{b^{2t}}{b^{2m}} \max\left(1, {m-t \choose d-1}\right) +  \frac{b^{2t}}{b^{2m}} \frac{1}{b(b^2-b)^d} (m-t+2)^{d-1} \right)\\
& \le & \frac{b^{2t}}{b^{2m}} \left(\frac{b^{3 d}}{b^2(b+1)^d} \max\left(1, {m-t \choose d-1}\right) +  \frac{1}{b} (m-t+2)^{d-1}\right)\\
& \le & \frac{b^{2t}}{b^{2m}} \left(\frac{b^{2 d}}{b^2} (m-t)^{d-1} +  \frac{1}{b} (m-t+2)^{d-1}\right)\\
& = &  \frac{b^{2t}}{b^{2m}} (m-t)^{d-1} \left(b^{2 (d-1)}  +  \frac{1}{b} \left(1+\frac{2}{m-t}\right)^{d-1}\right)\\
& \le &  \frac{b^{2t}}{b^{2m}} (m-t)^{d-1} \left(b^{2 (d-1)}  +  \frac{1}{b} 3^{d-1}\right)\\
& \le & \frac{b^{2t}}{b^{2m}} (m-t)^{d-1}  b^{2d}.
\end{eqnarray*}
\end{proof}

Now we can give the proof of Theorem~\ref{thm12}.

\begin{proof}[Proof of Theorem~\ref{thm12}]
Using \eqref{fo:l2perB} and Lemma~\ref{B_bound} we obtain
\begin{eqnarray*}
\EE[(L_{2,b^m}^{{\rm per}}(\widetilde{\cP}))^2] \le \frac{b^{2t}}{3^d}   \sum_{\emptyset \not=\uu \subseteq [d]} (m-t)^{|\uu|-1}  b^{2|\uu|} \le  b^{2t}  (m-t)^{d-1} \left(\frac{1+b^2}{3}\right)^d.
\end{eqnarray*}
This finishes the proof of Theorem~\ref{thm12}.
\end{proof}

Theorem~\ref{thm12} implies that the lower bound \eqref{lbd:Lper} is best possible.

\begin{cor}\label{co16}
Let $p \in [1,2]$. For every $m,d \in \NN$ with $m,d\ge 2$ and every prime number $b\ge d-1$ we have $$\inf_{\cP \subseteq [0,1)^d \atop |\cP|=N} L_{p,N}^{{\rm per}} (\cP) \lesssim_{d,b} (\log N)^{\frac{d-1}{2}}, \quad \quad \mbox{ where $N=b^m$.}$$ Hence, the lower bound from Corollary~\ref{cor:lowper} is best possible for all $p \in (1,2]$.
\end{cor}

\begin{proof}
It is known that for every $m,d \in \NN$, $m,d\ge 2$, and every prime number $b \ge d-1$ there exists a digital $(0,m,d)$-net $\cP$ in base $b$ (constructions are due to Faure~\cite{Fau} and Niederreiter~\cite{nie87}; see also \cite{DP10,LeoPi,niesiam}). According to Theorem~\ref{thm12} there exists a digital shift of depth $m$, such that the digitally shifted point set $\widetilde{\cP}$ has periodic $L_2$ discrepancy $$(L_{2,b^m}^{{\rm per}}(\widetilde{\cP}))^2  \le    \left(\frac{1+b^2}{3}\right)^d m^{d-1} .$$ Taking the square root yields the desired result.
\end{proof}

Note that it can be shown that a digitally shifted digital $(t,m,d)$-net over $\ZZ_b$ is a $(t,m,d)$-net in base $b$ with the same $t$-parameter, however, in general not a digital one (just adapt the proof of \cite[Lemma~3]{DP05a}). So the point set $\widetilde{\cP}$ appearing in the proof of Corollary~\ref{co16} is a $(0,m,d)$-net in base $b$.

\begin{rem}[Diaphony]\label{re:dia}\rm
We already mentioned the relation between periodic $L_2$ discrepancy and diaphony. As usual, see for example  \cite[Section~1.2.3]{DT}, we denote the diaphony by $F_N$, $$F_N(\cP):= \left(\sum_{\bsk \in \ZZ^d\setminus\{\bszero\}} \frac{1}{\rho(\bsk)^2} \left|\frac{1}{N} \sum_{n=0}^{N-1} \exp(2 \pi \icomp \bsk \cdot \bsx_n)\right|^2\right),$$ where for $\bsk=(k_1,\ldots,k_d)\in \ZZ^d$ we set $\rho(\bsk)=\prod_{j=1}^d \max(1,|k_j|)$. Then Theorem~\ref{thm12} also implies that for every $m,d \in \NN$ with $m,d \ge 2$ and every prime number $b\ge d-1$ there exists an $N=b^m$ element point set $\cP$ in $[0,1)^d$ such that 
\begin{equation}\label{bd:diaN}
F_N(\cP) \lesssim_d \frac{(\log N)^{\frac{d-1}{2}}}{N}.
\end{equation}
Previously, such a result was known only for $d=2$; see \cite{groz,X2000}. The bound \eqref{bd:diaN} is best possible in the order of magnitude in $N$. According to \cite{lev99} this has been first shown by Bykovsky~\cite{by} (however, we do not have access to this paper). Later proofs can be also found in \cite{HKP20,lev99}.
\end{rem}

\section{Summary and open problems}\label{sec:sum}

We summarize the main results about the order of magnitude of the minimal $L_p$ discrepancy in $N$ and for fixed dimension $d$.
\begin{itemize}
\item Star and extreme discrepancy: For every $p \in (1,\infty)$ and every $d,N \in \NN$ we have $$\inf_{\cP \subseteq [0,1)^d \atop |\cP|=N} L_{p,N}^{\bullet}(\cP) \asymp_{d,p} (1+\log N)^{\frac{d-1}{2}}, \quad \mbox{where } \bullet \in \{{\rm star, extr}\}.$$ For $d=2$ the result even holds true for $p=1$.
\item Periodic $L_p$ discrepancy: For the periodic $L_p$ discrepancy for every $p \in (1,\infty)$ and every $d,N \in \NN$ we have $$\inf_{\cP \subseteq [0,1)^d \atop |\cP|=N} L_{p,N}^{{\rm per}}(\cP) \gtrsim_{d,p} (1+\log N)^{\frac{d-1}{2}}.$$ Again, for $d=2$ the result holds true also for $p=1$. In the other direction so far we only have the following result. For every $p \in [1,2]$,  every $d,m \in \NN$ with $d,m \ge 2$ and every prime number $b\ge d-1$ we have 
\begin{equation}\label{sum:per}
\inf_{\cP \subseteq [0,1)^d \atop |\cP|=N} L_{p,N}^{{\rm per}}(\cP) \lesssim_{d} (1+\log N)^{\frac{d-1}{2}},\quad \mbox{where $N=b^m$.}
\end{equation}
It remains an open problem to extend \eqref{sum:per} to arbitrary $N \in \NN$ (i.e., not only prime powers) and to arbitrary $p \in [1,\infty)$. In this context we also mention the still open problem of finding for every $N\ge 2$ explicit constructions of $N$-point sets in $[0,1)^d$ with the optimal order of diaphony $F_N$, which is $(\log N)^{\frac{d-1}{2}}/N$.
\end{itemize}

Still there are also some open questions about relations between the three notions of $L_p$ discrepancy. In particular, let $\cP$ be an arbitrary $N$-point set in $[0,1)^d$. Then the following questions arise (see also \cite{HKP20}):
\begin{itemize}
\item Is is true that $L_{p,N}^{{\rm extr}}(\cP) \asymp_{p,d} L_{p,N}^{{\rm per}}(\cP)$? This is true for $d=1$. Note also  that $L_{p,N}^{{\rm extr}}(\cP) \le L_{p,N}^{{\rm per}}(\cP)$.
\item Is it true that $L_{p,N}^{{\rm per}}(\cP) \lesssim_{p,d} L_{p,N}^{{\rm star}}(\cP)$? Note that a corresponding inequality the other way round is not true in general.
\end{itemize}

\noindent {\bf Author's address:} Institute of Financial Mathematics and Applied Number Theory, Johannes Kepler University Linz, Austria, 4040 Linz, Altenberger Stra{\ss}e 69. Email: ralph.kritzinger@yahoo.de, friedrich.pillichshammer@jku.at

\end{document}